\documentclass{amsart}
\usepackage{amsmath,amsthm,amstext,amssymb,bbm,enumerate}

\usepackage[all]{xy}
\usepackage{float}

\theoremstyle{plain}
\newtheorem{theorem}{Theorem}[section]
\newtheorem{lemma}[theorem]{Lemma}
\newtheorem{question}[theorem]{Question}
\newtheorem{remark}[theorem]{Remark}
\newtheorem{corollary}[theorem]{Corollary}
\newtheorem{proposition}[theorem]{Proposition}
\theoremstyle{definition}
\newtheorem{definition}[theorem]{Definition}
\newtheorem{examples}[theorem]{Examples}

\newcommand{\betrag}[1]{\vert{#1}\vert}
\newcommand{\height}[1]{{\rm{ht}}(#1)}
\newcommand{\lub}{{\rm{lub}}}
\newcommand{\dom}[1]{{{\rm{dom}}(#1)}}

\newcommand{\otp}[1]{{{\rm{otp}}\left(#1\right)}}
\newcommand{\ran}[1]{{{\rm{ran}}(#1)}}
\newcommand{\supp}[1]{{{\rm{supp}}(#1)}}

\newcommand{\length}[2]{{\rm{lh}_{{#2}}}\left({#1}\right)}
\newcommand{\POT}[1]{{\mathcal{P}}({#1})}
\newcommand{\map}[3]{{#1}:{#2}\longrightarrow{#3}}
\newcommand{\Map}[5]{{#1}:{#2}\longrightarrow{#3};~{#4}\longmapsto{#5}}
\newcommand{\pmap}[4]{{#1}:{#2}\xrightarrow{#4}{#3}}
\newcommand{\Set}[2]{\left\{{#1}~\vert~{#2}\right\}}
\newcommand{\seq}[2]{\langle{#1}~\vert~{#2}\rangle}
\newcommand{\anf}[1]{{\text{``}\hspace{0.3ex}{#1}\hspace{0.3ex}\text{''}}}
\newcommand{\HH}[1]{{\rm{H}}(#1)}
\newcommand{\Add}[2]{{\rm{Add}}({#1},{#2})}
\newcommand{\id}{{\rm{id}}}
\newcommand{\Lim}{{\rm{Lim}}}

\newcommand{\LL}{{\rm{L}}}
\newcommand{\ZFC}{{\rm{ZFC}}}

\newcommand{\PPP}{{\mathbb{P}}}
\newcommand{\QQQ}{{\mathbb{Q}}}

\newcommand{\TTT}{{\mathbb{T}}}
\newcommand{\PP}{{\mathcal{P}}}
\newcommand{\VV}{{\rm{V}}}
\newcommand{\calC}{\mathcal{C}}
\newcommand{\calF}{\mathcal{F}}
\newcommand{\calU}{\mathcal{U}}
\newcommand{\CF}[2]{{\mathcal{D}}_{{#1}}^{{#2}}}

\begin{document}

\title[Large cardinals and layering]{Characterizing large cardinals in terms of layered posets}

\author{Sean Cox}
\address{
Department of Mathematics and Applied Mathematics \\
Virginia Commonwealth University \\
1015 Floyd Avenue \\
Richmond, Virginia 23284, USA 
}
\email{scox9@vcu.edu}

\author{Philipp L\"ucke}
\address{
Mathematisches Institut \\ 
Rheinische Friedrich-Wilhelms-Universit\"at Bonn \\ 
En\-de\-nicher Allee 60 \\
53115 Bonn \\
Germany 
}
\email{pluecke@uni-bonn.de}

\thanks{This work was partially supported by a grant from the Simons Foundation (\#318467 to Sean Cox) and a grant from the Deutsche Forschungsgemeinschaft (LU2020/1-1 to Philipp L\"ucke). The main results of this paper were obtained during a visit of the second author at VCU in March 2015. The second author would like to thank the first author for his hospitality during his stay in Richmond and the Simons Foundation for the financial support of his visit through the above grant. Finally, the authors would like to thank the anonymous referee for the careful reading of the manuscript.}


\begin{abstract}
 Given an uncountable regular cardinal $\kappa$, a partial order is $\kappa$-stationarily layered if the collection of regular suborders of $\PPP$ of cardinality less than $\kappa$ is stationary in $\PP_\kappa(\PPP)$. 
 We show that weak compactness can be characterized by this property of partial orders by proving that an uncountable regular cardinal $\kappa$ is weakly compact if and only if every partial order satisfying the $\kappa$-chain condition is $\kappa$-stationarily layered. 
 We prove a similar result for strongly inaccessible cardinals. 
Moreover, we show that the statement that all $\kappa$-Knaster partial orders are $\kappa$-stationarily layered implies that $\kappa$ is a Mahlo cardinal and every stationary subset of $\kappa$ reflects. This shows that this statement characterizes weak compactness in canonical inner models. In contrast, we show that it is also consistent that this statement holds at a non-weakly compact cardinal.  
\end{abstract}

\maketitle


\section{Introduction}\label{sec_Intro}

Since the results presented in this paper are motivated by classical questions on the \emph{productivity of chain conditions} in partial orders, we start with a short introduction to this topic (longer introduction can be found in \cite{MR1393943} and \cite{MR3271280}). 
Given an uncountable regular cardinal $\kappa$, we let $\calC_\kappa$ denote the statement that the product of two partial orders satisfying the $\kappa$-chain condition again satisfies the $\kappa$-chain condition. Then $\calC_\kappa$ implies the non-existence of $\kappa$-Souslin trees and $\mathrm{MA}_{\aleph_1}$ implies $\calC_{\aleph_1}$. In particular, the statement $\calC_{\aleph_1}$ is independent from the axioms of $\ZFC$. A folklore argument shows that $\calC_\kappa$ holds if $\kappa$ is weakly compact. A small modification of this argument yields the statement of the following proposition. 
Recall that, given an uncountable regular cardinal $\kappa$, a partial order is $\kappa$-Knaster if every $\kappa$-sized collection of conditions can be refined to a $\kappa$-sized  set of pairwise compatible conditions. Given an infinite cardinal $\nu$ and a sequence $\seq{\PPP_\gamma}{\gamma<\lambda}$, the corresponding  \emph{$\nu$-support product} consists of all elements $\vec{p}$ of the full product of these partial orders such that the set $\supp{\vec{p}\hspace{0.9pt}}=\Set{\gamma<\lambda}{\vec{p}(\gamma)\neq\mathbbm{1}_{\PPP_\gamma}}$ has cardinality at most $\nu$.

\begin{proposition}\label{proposition:WCyieldsCkappa}
 If $\kappa$ is a weakly compact cardinal and $\nu<\kappa$, then $\nu$-support products of partial orders satisfying the $\kappa$-chain condition are $\kappa$-Knaster. 
\end{proposition}

\begin{proof}
 Fix a sequence $\seq{\PPP_\gamma}{\gamma<\lambda}$ of partial orders satisfying the $\kappa$-chain condition and a sequence $\seq{\vec{p}_\alpha}{\alpha<\kappa}$ of conditions in the corresponding $\nu$-support product $\vec{\PPP}=\prod_{\gamma<\lambda}\PPP_\gamma$. With the help of the $\Delta$-system lemma, we can find $D\in[\kappa]^\kappa$ and a function  $\map{r}{\nu}{\lambda}$ such that the set $\Set{\supp{\vec{p}_\alpha}}{\alpha\in D}$ is a $\Delta$-system with root $\ran{r}$. Given $\alpha,\beta\in D$ with $\alpha<\beta$, there is a minimal $c(\alpha,\beta)\leq\nu$ with the property that either $c(\alpha,\beta)<\nu$ and the conditions $\vec{p}_\alpha(r(c(\alpha,\beta)))$ and $\vec{p}_\beta(r(c(\alpha,\beta)))$ are incompatible in $\PPP_{r(c(\alpha,\beta))}$, or $c(\alpha,\beta)=\nu$ and the conditions $\vec{p}_\alpha$ and $\vec{p}_\beta$ are compatible in $\vec{\PPP}$. By the weak compactness of $\kappa$, there is $H\in[D]^\kappa$ that is homogeneous for the resulting colouring $\map{c}{[D]^2}{\nu+1}$. Since each $\PPP_\xi$ satisfies the $\kappa$-chain condition, we can conclude that $c$ is constant on $[H]^2$ with value $\nu$ and this shows that the resulting sequence $\seq{\vec{p}_\alpha}{\alpha\in H}$ consists of pairwise compatible conditions in $\vec{\PPP}$. 
\end{proof}

The $\kappa$-Knaster property clearly implies the $\kappa$-chain condition. This property is typically used because of its nice product behavior: the product of two $\kappa$-Knaster partial orders is $\kappa$-Knaster, and the product of a $\kappa$-Knaster partial order  with a partial order satisfying the $\kappa$-chain condition satisfies the $\kappa$-chain condition. For reasons described next, we are interested in finding alternative proofs of the above proposition.

A series of deep results shows that, for regular cardinals $\kappa>\aleph_1$, many consequences of weak compactness can be derived from the assumption $\calC_\kappa$. In \cite{MR1318912}, Shelah showed that $\calC_\kappa$ fails if $\kappa$ is the successor of a singular cardinal. Rinot showed in \cite{MR3335852} that $\calC_\kappa$ implies that every stationary subset of $\kappa$ reflects. In particular, this result can be used to reprove a series of results of Shelah showing that $\calC_\kappa$ fails for all successors of uncountable regular cardinals. In addition, Rinot showed in \cite{MR3271280} that for $\kappa>\aleph_1$, the principle $\calC_\kappa$ implies a failure of $\square(\kappa)$ and therefore it implies that $\kappa$ is weakly compact in  G\"odel's constructible universe $\LL$. These results suggest an affirmative answer to the following question of Todor{\v{c}}evi{\'c}.

\begin{question}[Todor{\v{c}}evi{\'c}, {\cite[Question 8.4.27]{MR2355670}}]
 Are the following statements equivalent for every regular cardinal $\kappa>\aleph_1$?
 \begin{enumerate}
  \item $\kappa$ is weakly compact.

  \item $\calC_\kappa$ holds. 
 \end{enumerate}
\end{question}

In this paper, we want to consider properties of partial orders that imply the $\kappa$-chain condition, are preserved by forming products and are equivalent to the $\kappa$-chain condition if $\kappa$ is a weakly compact cardinal. Note that such properties provide alternative proofs of Proposition \ref{proposition:WCyieldsCkappa}. It is interesting to consider the question whether the $\kappa$-chain condition can be equivalent to such a property at non-weakly compact cardinals, because both possible answers yield interesting statements: a positive answer to this question would answer Todor{\v{c}}evi{\'c}'s question in the negative; while a negative answer leads to new characterizations of weak compactness using chain conditions.

We will now introduce the properties of partial orders studied in this paper. 
Remember that, given a partial order $\PPP$, we say that $\QQQ\subseteq\PPP$ is a regular suborder if the inclusion map preserves incompatibility and maximal antichains in $\QQQ$ are maximal in $\PPP$.

\begin{definition}
 Given a cardinal $\kappa$ and a partial order $\PPP$, we let $\mathrm{Reg}_\kappa(\PPP)$ denote the collection of all regular suborders of $\PPP$ of cardinality less than $\kappa$.  
\end{definition}

In the following, we will focus on properties of partial orders that imply that $\mathrm{Reg}_\kappa(\PPP)$ is \emph{large} in a certain sense. 
The following definition presents our main example of such a notion. It uses Jech's definition of stationarity in $\PP_\kappa(A)$ to express largeness (see Section \ref{sec_Prelims}).

\begin{definition}[{\cite[Definition 19]{Cox_Layerings}}]
 Given an uncountable regular cardinal $\kappa$, a partial order $\mathbb{P}$ is called \emph{$\kappa$-stationarily layered} if $\mathrm{Reg}_\kappa(\PPP)$ is stationary in $\PP_\kappa(\PPP)$.
\end{definition}

The notion of layering has appeared most frequently in the literature on saturated ideals.
Foreman, Magidor and Shelah used layered ideals in \cite{MR942519} to construct nonregular ultrafilters. Shelah also used the notion of a layered ideal in \cite{MR1026923} to address a variant of the \emph{Katetov question} asking about the existence of a topological space of size $\aleph_1$ that contains no isolated points and has the property that any real function has a point of continuity. 
The first author used the notion of layering in \cite{Cox_Layerings} to prove a general iteration theorem about so-called \emph{universal Kunen collapses}.

The following lemma is proved in \cite{Cox_Layerings}. For sake of completeness, we will present its short proof in Section \ref{sec_Prelims}.

\begin{lemma}[{\cite[Lemma 4]{Cox_Layerings}}]\label{lemma:StationaryLayeredKnaster}
 If $\kappa$ is an uncountable regular cardinal and $\PPP$ is a $\kappa$-stationarily layered partial order, then $\PPP$ is $\kappa$-Knaster. 
\end{lemma}

In particular, if every partial order that satisfies the $\kappa$-chain condition is $\kappa$-stationarily layered, then this lemma shows that $\calC_\kappa$ holds.

In the spirit of the approach outlined above, we are interested in classes of stationarily layered partial orders that are closed under products. The following definition yields examples of such classes. These classes will have the property that any two members are layered on a \emph{common stationary set} (in a sense made precise later) and this common layering can be used to show that certain classes of this form  are also closed under products with larger supports.

\begin{definition}
 Let $\kappa$ be an uncountable regular cardinal, $\lambda\geq\kappa$ be a cardinal and $\calF$ be a filter on $\PP_\kappa(\lambda)$. 
 \begin{enumerate}
  \item A partial order $\PPP$  is $\calF$-layered, if it has cardinality at most $\lambda$ and  $$\Set{a\in\PP_\kappa(\lambda)}{s[a]\in\mathrm{Reg}_\kappa(\PPP)}\in\calF$$ holds for every surjection $\map{s}{\lambda}{\PPP}$. 

  \item A partial order $\PPP$ is \emph{completely $\calF$-layered} if every subset of $\PPP$ of cardinality at most $\lambda$ is contained in a regular suborder of $\PPP$ of cardinality at most $\lambda$ and every regular suborder of $\PPP$ of size at most $\lambda$ is $\calF$-layered. 
 \end{enumerate} 
\end{definition}

In Section \ref{section:FLayeredPoset}, we will show that, if $\calF$ is a normal filter on $\PP_\kappa(\lambda)$, then every completely $\calF$-layered partial order is $\kappa$-stationarily layered and the class of partial orders with this property is closed under finite support products. Moreover, for many interesting filters $\calF$, the class of completely $\calF$-layered partial orders is also closed under products with larger supports. Finally, if $\kappa$ is weakly compact and $\calF_{wc(\kappa)}$ is the \emph{weakly compact filter on $\PP_\kappa(\kappa)$} (see Example  \ref{example:Relevant Filters}.(\ref{example:item:WCFilter})), then the class of completely $\calF_{wc(\kappa)}$-layered partial orders will be shown to closed under $\nu$-support products for every $\nu<\kappa$ and the following theorem proven in Section \ref{section:FilterLargeCardinals} shows that this class is equal to the class of all partial orders satisfying the $\kappa$-chain condition.

\begin{theorem}\label{theorem:WeaklyCompactEquivalentChainCondition}
 Given a weakly compact cardinal $\kappa$, the following statements are equivalent for every partial order $\PPP$: 
 \begin{enumerate}
  \item $\PPP$ satisfies the $\kappa$-chain condition. 

  \item $\PPP$ is $\kappa$-Knaster. 

  \item $\PPP$ is $\kappa$-stationarily layered. 

  \item $\PPP$ is completely $\calF_{wc(\kappa)}$-layered. 
 \end{enumerate}
\end{theorem}

The results mentioned above show that complete $\calF_{wc(\kappa)}$-layeredness satisfies the demands listed above and it is interesting to ask whether the existence of a normal uniform filter $\calF$ on $\PP_\kappa(\lambda)$ with the property that the $\kappa$-chain condition is equivalent to complete $\calF$-layeredness implies the weak compactness of $\kappa$. It turns out that the weak compactness of $\kappa$ already follows from the weaker assumption that the  $\kappa$-chain condition is equivalent to stationary layeredness. This leads to the following new characterization of weakly compact cardinals proven in Section \ref{section:CharacterizeLarge}.

\begin{theorem}\label{thm_WeakCompactEquiv}
 The following statements are equivalent for every  uncountable regular cardinal $\kappa$:
 \begin{enumerate}
  \item $\kappa$ is weakly compact. 

   \item Every partial order of cardinality at most $\kappa$ satisfying the $\kappa$-chain condition is $\kappa$-stationarily layered.

  \item Every partial order satisfying the $\kappa$-chain condition is $\kappa$-stationarily layered.  
\end{enumerate}
\end{theorem}

The proof of the above theorem also provides a new characterization of inaccessibility in terms of stationary layeredness.   
This characterization makes use of the following strengthening of the Knaster property.

\begin{definition}
 Given an uncountable cardinal $\kappa$, a partial order $\PPP$ is \emph{${<}\kappa$-linked} if there is $\lambda<\kappa$ and a function $\map{c}{\PPP}{\lambda}$ that is injective on antichains in $\PPP$. 
\end{definition}

\begin{theorem}\label{thm_InaccEquiv}
 The following statements are equivalent for every  uncountable regular cardinal $\kappa$: 
 \begin{enumerate}
  \item\label{item_Inacc} $\kappa$ is strongly inaccessible. 

  \item\label{item_LessKappaLinkedEquivLayer} Every ${<}\kappa$-linked partial order of cardinality at most $\kappa$ is $\kappa$-stationarily layered. 

   \item Every ${<}\kappa$-linked partial order is $\kappa$-stationarily layered.
 \end{enumerate}
\end{theorem}

In light of the above results, it is natural to ask if there is a large cardinal property that corresponds to the statement that every $\kappa$-Knaster partial order is $\kappa$-stationarily layered. The following results show that the question whether this equality characterizes weak compactness is independent from the axioms of set theory. We start by presenting results showing that this assumption yields large cardinal properties intermediate between inaccessibility and weak compactness. The proof of the first part of the following result makes use of Theorem \ref{thm_InaccEquiv} and a theorem of Todor{\v{c}}evi{\'c} from \cite{MR908147} that characterizes Mahlo cardinals through the existence of \emph{special $\kappa$-Aronszajn trees} (see Section \ref{section:Knaster}). The second part of the theorem is proved using Todor{\v{c}}evi{\'c}'s method of \emph{walks on ordinals}. The second conclusion shows that the above assumption has consistency strength greater than the existence of a Mahlo cardinal. Remember that, given an uncountable regular cardinal $\kappa$, a stationary subset $S$ of $\kappa$ \emph{reflects} if there is an $\alpha\in\kappa\cap\Lim$ with the property that $S\cap\alpha$ is stationary in $\alpha$.

\begin{theorem}\label{theorem:KnasterStationary} 
 If $\kappa$ is an uncountable regular cardinal with the property that every $\kappa$-Knaster partial order of cardinality at most $\kappa$ is $\kappa$-stationarily layered, then the following statements hold: 
 \begin{enumerate}
  \item[(i)] $\kappa$ is a Mahlo cardinal. 

  \item[(ii)] Every stationary subset of $\kappa$ reflects. 
 \end{enumerate}
\end{theorem}

Seminal results of Jensen (see {\cite[Section 6]{MR0309729}}) show that the above reflection property characterizes weak compactness in G\"odel's constructible universe $\LL$. These result were extended by Zeman in \cite{MR2563821} to a much larger class of canonical inner models. Together with {\cite[Corollary 0.2.]{MR2563821}}, the above theorems yields the following result.

\begin{corollary}\label{corollary:KnasterLayeredExtenderModels}
 Let $\LL[E]$ be a Jensen-style extender model. In $\LL[E]$, the following statements are equivalent for every uncountable regular cardinal $\kappa$: 
 \begin{enumerate}
  \item $\kappa$ is weakly compact. 

  \item Every $\kappa$-Knaster partial order of cardinality at most $\kappa$ is $\kappa$-stationarily layered. \qed
 \end{enumerate}
\end{corollary}

In contrast, we will show that consistently there can be a non-weakly compact cardinal with the property that every $\kappa$-Knaster partial order is $\kappa$-stationarily layered. The proof of the following theorem shows that such cardinals exist in a model constructed by Kunen in \cite{MR495118}.

\begin{theorem}\label{theorem:KunenModelNonWeaklyCompactKnasterLayered}
 If $\kappa$ is a weakly compact cardinal, then there is a partial order $\PPP$ such that the following statements hold in $\VV[G]$ whenever $G$ is $\PPP$-generic over $\VV$. 
 \begin{enumerate}
  \item $\kappa$ is inaccessible and not weakly compact. 

  \item Every $\kappa$-Knaster partial order is $\kappa$-stationarily layered. 

  \item For every $\nu<\kappa$, the class of $\kappa$-Knaster partial orders is closed under $\nu$-support products.
 \end{enumerate} 
\end{theorem}

We outline the structure of this paper. Section \ref{sec_Prelims} includes the relevant background material. In Section \ref{section:FLayeredPoset}, we show that classes of completely $\calF$-layered partial orders satisfy the requirements listed above. Section \ref{section:FilterLargeCardinals} contains the proof of Theorem \ref{theorem:WeaklyCompactEquivalentChainCondition} and an analogous result for supercompact cardinals. The characterizations of inaccessible and weakly compact cardinals stated above are proven in Section \ref{section:CharacterizeLarge}. Moreover, we will prove Theorem \ref{theorem:KunenModelNonWeaklyCompactKnasterLayered} in this section. Section \ref{section:Knaster} contains the proof of Theorem \ref{theorem:KnasterStationary}. In Section \ref{sec_Questions}, we conclude the paper with some questions raised by the above results.


\section{Preliminaries}\label{sec_Prelims}

We start by recalling definitions and previous results relevant for the proofs of the above theorems.

A map $\map{e}{\PPP}{\QQQ}$ between partial orders is called a \emph{regular embedding} if $e$ is order and incompatibility preserving, and whenever $A$ is a maximal antichain in $\PPP$ then $e[A]$ is a maximal antichain in $\QQQ$. 
The following well-known observation provides an alternative characterization of regular embeddings that is used throughout this paper. Its short proof can be found, for example, in {\cite[Section 2]{Cox_Layerings}}.

\begin{proposition}
 The following statements are equivalent for order and incompatibility preserving map $\map{e}{\PPP}{\QQQ}$ between partial orders: 
 \begin{enumerate}
  \item $e$ is a regular embedding. 

  \item For every $q\in\QQQ$, there is $p\in\PPP$ such that for all $p^\prime\leq_\PPP p$, the conditions $e(p^\prime)$ and $q$ are compatible in $\QQQ$. 
 \end{enumerate}
\end{proposition}

In the above situation, the (possibly non-unique) condition $p$ is called an \emph{$e$-reduct of $q$ into $\mathbb{P}$}. 
Given partial orders $\PPP\subseteq\QQQ$, we say that $\PPP$ is a \emph{suborder} of $\QQQ$ if the inclusion map $\id_\PPP$ is incompatibility preserving and we say that $\PPP$ is a  \emph{regular suborder} of $\QQQ$ if $\id_\PPP$ is a regular embedding. 
In the later case, we just say \emph{reduct} instead of $\id_\PPP$-reduct.  The variations of the following basic observation will be used throughout this paper.

\begin{proposition}\label{proposition:REgularClosureChainCondition}
 Let $\kappa$ be an uncountable regular cardinal and let $\lambda\geq\kappa$ be a cardinal with $\lambda=\lambda^{{<}\kappa}$. If $\QQQ$ is a partial order that satisfies the $\kappa$-chain condition,  $P$ is a subset of $\QQQ$ of cardinality at most $\lambda$ and $C$ is a club in $\PP_\kappa(\QQQ)$, then there is a regular suborder $\PPP$ of $\QQQ$ of cardinality at most $\lambda$ such that $P\subseteq\PPP$ and $C\cap\PP_\kappa(\PPP)$ is club in $\PP_\kappa(\PPP)$.  
\end{proposition}

\begin{proof}
 By our assumptions, there is an increasing continuous sequence $\seq{\PPP_\alpha}{\alpha\leq\kappa}$ of suborders of $\QQQ$ of cardinality at most $\lambda$ such that the following statements hold for all $\alpha<\kappa$: 
 \begin{enumerate}
  \item $P\subseteq\PPP_\alpha$. 

  \item If $A$ is a maximal antichain in $\PPP_\alpha$, then there is a maximal antichain $\bar{A}$ in $\QQQ$ with $A\subseteq\bar{A}\subseteq\PPP_{\alpha+1}$. 

  \item If $a\in\PP_\kappa(\PPP_\alpha)$, then there is $c\in C$ with $a\subseteq c\subseteq\PPP_{\alpha+1}$. 
 \end{enumerate}

 Let $A$ be a maximal antichain in $\PPP_\kappa$. Since $\PPP_\kappa$ is a suborder of $\QQQ$, we know that $A$ is an antichain in $\QQQ$ and $\betrag{A}<\kappa$. Then there is $\alpha<\kappa$ with $A\subseteq\PPP_\alpha$. Since $A$ is a maximal antichain in $\PPP_\alpha$, there is a maximal antichain $\bar{A}$ in $\QQQ$ with $A\subseteq\bar{A}\subseteq\PPP_{\alpha+1}\subseteq\PPP_\kappa$. Then $\bar{A}$ is an antichain in $\PPP_\kappa$ and we can conclude that $A=\bar{A}$ is also a maximal antichain in $\QQQ$. This shows that $\PPP_\kappa$ is a regular suborder of $\QQQ$. Finally, our construction and the regularity of $\kappa$ ensure that $C\cap\PP_\kappa(\PPP)$  is unbounded in $\PP_\kappa(\PPP_\kappa)$. This yields the statement of the proposition, because $C\cap\PP_\kappa(\PPP)$  is obviously closed in $\PP_\kappa(\PPP_\kappa)$. 
\end{proof}

In this paper, we use Jech's notion of stationarity in $\PP_\kappa(A)$: a subset $S$ of $\PP_\kappa(A)$ is \emph{stationary in $\PP_\kappa(A)$} if it meets every subset of $\PP_\kappa(A)$ which is $\subseteq$-continuous and cofinal in $\PP_\kappa(A)$. 
We use this notion of stationarity rather than the generalized version of Shelah, because the stronger notion seems to be needed to prove Lemma \ref{lemma:StationaryLayeredKnaster} and Lemma \ref{lemma:SpecialAronszajnNotLayered}. If $\kappa = \aleph_1$, then both notions of stationarity are equivalent, but for $\kappa > \aleph_1$ certain instances of Chang's Conjecture may cause them to differ (see \cite{MR2768692}).

The next lemma provides useful characterizations of stationary layeredness. In many arguments, we will just need the assumption that there is some regular cardinal $\theta>\kappa$ and some elementary submodel $M$ of $\HH{\theta}$ with the above properties to derive certain consequence of stationary layeredness. The following lemma shows that this assumption is equivalent to stationary layeredness.

\begin{lemma}\label{lemma:StationaryLayeredElementarySubmodel}
The following statements are equivalent for every uncountable regular cardinal $\kappa$ and every partial order $\PPP$: 

 \begin{enumerate}
  \item[(i)] $\PPP$ is $\kappa$-stationarily layered. 

  \item[(ii)] For every regular cardinal $\theta>\kappa$ with $\PPP\in\HH{\theta}$, the collection of all elementary substructures $M$ of $\HH{\theta}$ with $\betrag{M}<\kappa$, $\kappa\cap M\in \kappa$ and $\PPP\cap M\in\mathrm{Reg}_\kappa(\PPP)$ is stationary in $\PP_\kappa(\HH{\theta})$. 

  \item[(iii)] There is a regular cardinal $\theta$ with $\PP_\kappa(\PP)\in\HH{\theta}$ and an elementary substructure $M$ of $\HH{\theta}$ with $\betrag{M}<\kappa$, $\PPP\in M$, $\kappa\cap M\in \kappa\in M$ and $\PPP\cap M\in\mathrm{Reg}_\kappa(\PPP)$.  
 \end{enumerate} 
\end{lemma}

\begin{proof}
 The equivalence between (i) and (ii) follows from standard arguments on lifting and projecting stationary sets (see {\cite[Section 1]{MR0357121}}) and the fact that the collection of all elementary submodels $M$ of $\HH{\theta}$ with $\betrag{M}<\kappa$ and $\kappa\cap M\in\kappa$ forms a club in $\PP_\kappa(\HH{\theta})$.

 Now, assume that $\mathrm{Reg}_\kappa(\PPP)$ is not stationary in $\PP_\kappa(\PPP)$. By the results of {\cite[Section 1]{MR0357121}}, there is a function $\map{F}{\PP_\omega(\PPP)}{\PP_\kappa(\PPP)}$ such that $a\notin\mathrm{Reg}_\kappa(\PPP)$ holds for every $a\in\PP_\kappa(\PPP)$ with $F[\PP_\omega(a)]\subseteq\PP_\kappa(a)$. Let $\theta$ and $M$ be as above. Then $F\in\HH{\theta}$. By elementarity, the model $M$ contains a function $\map{F_M}{\PP_\omega(\PPP)}{\PP_\kappa(\PPP)}$ with the above properties. Given $e\in\PP_\omega(\PPP\cap M)$, we have $e\in M$, $F_M(e)\in M$ and therefore $F_M(e)\subseteq\PPP\cap M$. By the properties of $F_M$, this implies that $\PPP\cap M\notin\mathrm{Reg}_\kappa(\PPP)$. 
\end{proof}

Using Lemma \ref{lemma:StationaryLayeredElementarySubmodel}, it is possible to give a short prove of Lemma \ref{lemma:StationaryLayeredKnaster}.  The following proof differs somewhat from the proof presented in \cite{Cox_Layerings}.

\begin{proof}[Proof of Lemma \ref{lemma:StationaryLayeredKnaster}]
 Let $\vec{p}=\seq{p_\alpha}{\alpha<\kappa}$ be an injective enumeration of conditions in $\PPP$ and let $\theta>\kappa$ be a regular cardinal with $\PPP\in\HH{\theta}$. By Lemma \ref{lemma:StationaryLayeredElementarySubmodel}, there is a stationary subset $S$ of $\PP_\kappa(\HH{\theta})$ consisting of elementary substructures $M$ of $\HH{\theta}$ with $\vec{p}\in M$, $\kappa\cap M\in \kappa$ and $\PPP\cap M\in\mathrm{Reg}_\kappa(\PPP)$. Given $M\in S$, there is a reduct $r(M)$ of $p_{\kappa\cap M}$ into $\PPP\cap M$. This defines a regressive function $\map{r}{S}{\HH{\theta}}$ and we can find $q\in\PPP$ and $S^\prime\subseteq S$ stationary in $\PP_\kappa{\HH{\theta}}$ such that $r(M)=q\in M$ for all $M\in S^\prime$. Pick $M,N\in S^\prime$ with $\kappa\cap M<\kappa\cap N$. Then the conditions $p_{\kappa\cap M}$ and $q$ are compatible in $\PPP$ and, since $p_{\kappa\cap M},q\in N$, there is $q^\prime\in\PPP\cap N$ extending $p_{\kappa\cap M}$ and $q$. Then the conditions $p_{\kappa\cap N}$ and $q^\prime$ are compatible in $\PPP$ and hence the conditions $p_{\kappa\cap M}$ and $p_{\kappa\cap N}$ are compatible in $\PPP$. This shows that the sequence $\seq{p_{\kappa\cap M}}{M\in S^\prime}$ consists of $\kappa$-many pairwise compatible conditions. 
\end{proof}


\section{$\calF$-layered partial orders}\label{section:FLayeredPoset}

Throughout this section, $\kappa$ denotes an uncountable regular cardinal, $\lambda$ denotes a cardinal greater than or equal to $\kappa$ and $\calF$ denotes a normal filter on $\PP_\kappa(\lambda)$. Given a cardinal $\nu<\kappa$ and a function $\map{\varphi}{{}^\nu\lambda}{\lambda}$, we define $$\mathrm{Cl}(\varphi) ~ = ~\Set{a\in\PP_\kappa(\lambda)}{\varphi[{}^\nu a]\subseteq a}.$$

In the following, we derive basic structural properties of classes of $\calF$-layered partial orders. These results show that, for many choices of $\calF$, the class of completely $\calF$-layered partial orders satisfies the properties listed in Section \ref{sec_Intro}. We start by presenting the main examples of such filters.

\begin{examples}\label{example:Relevant Filters}
 \begin{enumerate}[(1)]
  \item We let $\CF{\kappa,\lambda}{}$ denote the \emph{club filter} on $\PP_\kappa(\lambda)$. Then $\CF{\kappa,\lambda}{}$ is normal (and therefore  ${<}\kappa$-closed) and it contains all sets of the form $\mathrm{Cl}(\varphi)$ with $\map{\varphi}{{}^n\lambda}{\lambda}$ for some $n<\omega$.

  \item  Let $\nu<\kappa$ be an infinite cardinal with $\lambda=\lambda^\nu$ and $\mu^\nu<\kappa$ for all $\mu<\kappa$. We define the \emph{$\nu$-club filter $\CF{\kappa,\lambda}{\nu}$ on $\PP_\kappa(\lambda)$} to be the collection of all subsets of $\PP_\kappa(\lambda)$ that contain a subset of the form $\mathrm{Cl}(\varphi)$ with $\map{\varphi}{{}^\nu\lambda}{\lambda}$.

Fix a bijection $\map{\sigma}{{}^\nu\lambda}{\lambda}$. Then $\mathrm{Cl}(\sigma)\cap\mathrm{Cl}(\varphi\circ\sigma^{{-}1})\subseteq\mathrm{Cl}(\varphi)$ holds for every   $\map{\varphi}{{}^\nu\lambda}{\lambda}$. Since sets of the form $\mathrm{Cl}(\varphi\circ\sigma^{{-}1})$ are club in $\PP_\kappa(\lambda)$, this shows that $\CF{\kappa,\lambda}{\nu}$ is equal to the filter induced by the restriction of $\CF{\kappa,\lambda}{}$ to the stationary set $\mathrm{Cl}(\sigma)$ and therefore $\CF{\kappa,\lambda}{\nu}$ is a normal filter on $\PP_\kappa(\lambda)$.

  \item\label{example:item:WCFilter} Assume that $\kappa$ is weakly compact. By the results of \cite{MR0281606}, there is a normal filter on $\kappa$ that is generated by sets of the form $$R_{\Phi,A,a} ~ = ~ \Set{\alpha<\kappa}{\VV_\alpha\models\Phi(A\cap\VV_\alpha,a)},$$ where $\Phi$ is a $\Pi^1_1$-formula, $A\subseteq\VV_\kappa$, $a\in\VV_\kappa$ and $\VV_\kappa\models\Phi(A,a)$. The \emph{weakly compact filter $\calF_{wc(\kappa)}$ on $\PP_\kappa(\kappa)$} is the filter generated by sets of the above form viewed as elements of $\PP_\kappa(\kappa)$. This filter is again normal and contains the collection of all inaccessible cardinals less than $\kappa$ as an element. Finally, for every function $\map{\varphi}{{}^\nu\kappa}{\kappa}$ with $\nu<\kappa$, we can find $\langle \Phi,A,a\rangle$ as above such that $\varphi[{}^\nu\mu]\subseteq\mu$ for every $\mu\in R_{\Phi,A,a}$. This shows that $\calF_{wc(\kappa)}$ extend the $\nu$-club filter on $\PP_\kappa(\kappa)$ for every $\nu<\kappa$.

  \item Assume that $\kappa$ is $\lambda$-supercompact. Then there is a normal fine ultrafilter $\calU$ on $\PP_\kappa(\lambda)$. Let $\map{j_\calU}{\VV}{N}$ be the corresponding ultrapower embedding. Then $\calU=\Set{S \subseteq \PP_\kappa(\lambda)}{j_\calU[\lambda]\in j_\calU(S)}$. We show that $\calU$ extends the $\nu$-club filter on $\PP_\kappa(\lambda)$ for every $\nu<\kappa$: Fix $\map{\varphi}{{}^\nu\lambda}{\lambda}$ with $\nu<\kappa$. If $\map{f}{\nu}{j_\calU[\lambda]}$, then there is $\map{f_0}{\nu}{\lambda}$ with $f=j_\calU\circ f_0$ and this implies that $$j_\calU(\varphi)(f) ~ = ~ j_\calU(\varphi)(j_\calU(f_0)) ~ = ~ j_\calU(\varphi(f_0)) ~ \in ~ j_\calU[\lambda].$$ We can conclude that $j_\calU[\lambda]\in j_\calU(\mathrm{Cl}(\varphi))$ and this shows that $\mathrm{Cl}(\varphi)\in\calU$.  
 \end{enumerate} 
\end{examples}

The proof of Theorem \ref{theorem:KunenModelNonWeaklyCompactKnasterLayered} in Section \ref{section:FilterLargeCardinals} contains another example of a normal filter on $\PP_\kappa(\kappa)$ with the property that the resulting class of completely layered partial orders satisfies the requirements listed in Section \ref{sec_Intro}. In the following, we prove the statements about $\calF$-layered partial orders mentioned in Section \ref{sec_Intro}.

\begin{proposition}\label{proposition;CharacterizeFLayered}
  The following statements are equivalent for every partial order of cardinality at most $\lambda$: 
 \begin{enumerate}
  \item $\PPP$ is $\calF$-layered.

  \item There is a surjection $\map{s}{\lambda}{\PPP}$ with $\Set{a\in\PP_\kappa(\lambda)}{s[a]\in\mathrm{Reg}_\kappa(\PPP)}\in\calF$.
 \end{enumerate}
\end{proposition}

\begin{proof}
 Given two surjections $\map{s_0,s_1}{\lambda}{\PPP}$, the set $\Set{a\in\PP_\kappa(\lambda)}{s_0[a]=s_1[a]}$ is an element of $\CF{\kappa,\lambda}{}$ and, since $\calF$ is normal, it is contained in $\calF$.  
\end{proof}

\begin{lemma}\label{proposition:FlayeredStationaryLayered}
  Every $\calF$-layered partial order is $\kappa$-stationarily layered.  
\end{lemma}

\begin{proof}
 Let $\map{s}{\lambda}{\PPP}$ be a surjection and $\theta>\lambda$ be a sufficiently large regular cardinal. Since the normality of $\calF$ implies  that every element of $\calF$ is stationary in $\PP_\kappa(\lambda)$, we can use a standard lifting argument to show that the collection $S$ of all elementary submodels $M$ of $\HH{\theta}$ in $\PP_\kappa(\HH{\theta})$ such that $s\in M$, $\kappa\cap M\in\kappa$ and $s[\lambda\cap M]\in\mathrm{Reg}_\kappa(\PPP)$ is stationary in $\PP_\kappa(\lambda)$. Elementarity implies that $\PPP\cap M=s[\lambda\cap M]$ holds for all $M\in S$ and, by Lemma  \ref{lemma:StationaryLayeredElementarySubmodel}, we can conclude that $\PPP$ is $\kappa$-stationarily layered.   
\end{proof}

\begin{lemma}\label{lemma:FLayeredProducts}
 Let $\nu<\kappa$ be a cardinal such that $\lambda=\lambda^\nu$, 
$\mu^\nu<\kappa$ for every $\mu<\kappa$ and $\mathrm{Cl}(\varphi)\in\calF$ for every function $\map{\varphi}{{}^\nu\lambda}{\lambda}$. Then the class of $\calF$-layered partial orders is closed under $\nu$-products of length $\lambda$. 
\end{lemma}

\begin{proof}
 Let $\seq{\PPP_\delta}{\delta<\lambda}$ be a sequence of $\calF$-layered partial orders and let $s$ be a surjection of $\lambda$ onto the corresponding $\nu$-support product ${\vec{\PPP}=\prod_{\delta<\lambda}\PPP_\delta}$. Fix a sequence $\seq{d_\gamma\in{}^\nu\lambda}{\gamma<\lambda}$ of injections such that $\supp{s(\gamma)}\subseteq\ran{d_\gamma}$ for every $\gamma<\lambda$. Given $\delta<\lambda$, let $\map{s_\delta}{\lambda}{\PPP_\delta}$ denote the induced surjection defined by  $s_\delta(\gamma)=s(\gamma)(\delta)$ for all $\gamma<\lambda$. In this situation, our assumptions ensures that the set $F_\delta=\Set{a\in\PP_\kappa(\lambda)}{s_\delta[a]\in\mathrm{Reg}_\kappa(\PPP)}$ is an element of $\calF$ for every $\delta<\lambda$.

Pick  $\map{\varphi}{\lambda\times\lambda}{\lambda}$ with $\varphi(\gamma,\xi)=d_\gamma(\xi)$ for all $\gamma<\lambda$ and $\xi<\nu$. In addition, pick $\map{\psi}{{}^\nu\lambda\times{}^\nu\lambda}{\lambda}$ such that for every map  $\map{f}{\nu}{\lambda}$ and every injection $\map{g}{\nu}{\lambda}$, we have $\supp{s(\psi(f,g))}\subseteq\ran{g}$ and $s(\psi(f,g))(g(\xi))=s_{g(\xi)}(f(\xi))$ for all $\xi<\nu$.  Let $C$ denote the club in $\PP_\kappa(\lambda)$ consisting of all $a\in\PP_\kappa(\lambda)$ with $\nu\subseteq a$ and define $$F ~ = ~ C ~ \cap ~ \mathrm{Cl}(\varphi) ~ \cap ~ \mathrm{Cl}(\psi) ~ \cap ~ \Delta_{\delta<\lambda}F_\delta ~ \in ~ \calF.$$

Pick $a\in F$ and $\gamma_0,\gamma_1\in a$ such that the corresponding conditions $\vec{q}_0=s(\gamma_0)$ and $\vec{q}_1=s(\gamma_1)$ are compatible in $\vec{\PPP}$. Since $\nu\subseteq a$ and $\varphi[a\times\nu]\subseteq a$, we have $\supp{\vec{q}_i}\subseteq a$ for all $i<2$ and there is an injection $\map{g}{\nu}{a}$ with $\supp{\vec{q}_0}\cup\supp{\vec{q}_1}\subseteq\ran{g}$. Given $\xi<\nu$, we have $g(\xi)\in a\in F_{g(\xi)}$ and hence $s_{g(\xi)}[a]\in\mathrm{Reg}_\kappa(\PPP_{g(\xi)})$. This shows that there is a function $\map{f}{\nu}{a}$ such that $s_{g(\xi)}(f(\xi))\leq_{\PPP_{g(\xi)}}\vec{q}_i(g(\xi))$ for all $\xi<\nu$ and $i<2$. Set $\vec{q}=s(\varphi(f,g))\in s[a]\subseteq\vec{\PPP}$. By the above choices, $\vec{q}$ is a common extension of $\vec{q}_0$ and $\vec{q}_1$ in $s[a]$. This shows that $s[a]$ is a suborder of $\vec{\PPP}$.

 Next, pick $a\in F$ and a condition $\vec{p}$ in $\vec{\PPP}$. Then there is an injection $\map{g}{\nu}{a}$ with $a\cap\supp{\vec{p}\hspace{0.9pt}}\subseteq\ran{g}$. As above, we have $g(\xi)\in a\in F_{g(\xi)}$ and $s_{g(\xi)}[a]\in\mathrm{Reg}_\kappa(\PPP_{g(\xi)})$ for every $\xi<\nu$. This shows that there is a function $\map{f}{\nu}{a}$ such that $s_{g(\xi)}(f(\xi))$ is a reduct of $\vec{p}(g(\xi))$ into $s_{g(\xi)}[a]$ for every $\xi<\nu$. Define $\vec{q}=s(\varphi(f,g))\in s[a]\subseteq\vec{\PPP}$ and let $\vec{r}$ be a condition in $s[a]$ below $\vec{q}$. As above, $\nu\subseteq a$ and $\varphi[a\times\nu]\subseteq a$ imply that $\supp{\vec{r}\hspace{0.9pt}}\subseteq a$. Given $\xi<\nu$, we have $$\vec{r}(g(\xi)) ~\leq_{\PPP_{g(\xi)}} ~ \vec{q}(g(\xi)) ~ = ~ s(\varphi(f,g))(g(\xi)) ~ = ~ s_{g(\xi)}(f(\xi))$$ and this implies that $\vec{r}(g(\xi))$ is compatible with $\vec{p}(g(\xi))$ in $\PPP_{g(\xi)}$. Since we ensured that  $\supp{\vec{p}\hspace{0.9pt}}\cap\supp{\vec{r}\hspace{0.9pt}}\subseteq\ran{g}$, we can conclude that $\vec{r}$ is compatible with $\vec{p}$ in $\vec{\PPP}$.

The above computations show  that  $s[a]$ is a regular suborder of $\vec{\PPP}$ for every $a\in F$. Since $F\in\calF$, this yields the statement of the lemma.  
\end{proof}

\begin{corollary}
 \begin{enumerate}
  \item The class of $\calF$-layered partial orders is closed under finite support products of length $\lambda$.

  \item Let $\nu<\kappa$ be an infinite cardinal with $\lambda=\lambda^\nu$ and $\mu^\nu<\kappa$ for all $\mu<\kappa$. If $\calF$ extends the $\nu$-club filter on $\PP_\kappa(\lambda)$, then the class of $\calF$-layered partial orders is closed under $\nu$-support products of length $\lambda$.

  \item If $\kappa$ is a weakly compact cardinal, then the class of $\calF_{wc(\kappa)}$-layered partial orders is closed under $\nu$-products of length $\kappa$ for every $\nu<\kappa$. 

  \item If $\calU$ is a normal fine ultrafilter on $\PP_\kappa(\lambda)$ and $\lambda=\lambda^{{<}\kappa}$, then the class of $\calU$-layered partial orders is closed under $\nu$-support products of length $\lambda$ for every $\nu<\kappa$.  
\qed
 \end{enumerate}
\end{corollary}

In the second part of this section, we prove the analogs of the above results for completely $\calF$-layered partial orders. We start by showing that the two concepts coincide on the class of partial orders of size at most $\lambda$.

\begin{lemma}
 Every $\calF$-layered partial order is completely $\calF$-layered. 
\end{lemma}

\begin{proof}
 Let $\QQQ$ be an $\calF$-layered partial order and $\PPP$ be a regular suborder of $\QQQ$. Fix a surjection $\map{\bar{s}}{\lambda}{\QQQ}$ and a function $\map{f}{\lambda}{\lambda}$ with the property that $\bar{s}(f(\gamma))$ is a reduct of $\bar{s}(\gamma)$ into $\PPP$ for all $\gamma<\lambda$ and $f(\gamma)=\gamma$ for all $\gamma<\lambda$ with $\bar{s}(\gamma)\in\PPP$. Define $\map{s=\bar{s}\circ f}{\lambda}{\PPP}$ and fix a function $\map{g}{\lambda\times\lambda}{\lambda}$ with the property that $s(g(\gamma_0,\gamma_1))\leq_\PPP {s(\gamma_0),s(\gamma_1)}$ for all $\gamma_0,\gamma_1\in S$ with $s(\gamma_0)$ and $s(\gamma_1)$ compatible in $\PPP$.

 By our assumptions, there is an $F\in\calF$ such that $\bar{s}[a]\in\mathrm{Reg}_\kappa(\QQQ)$, $f[a]\subseteq a$ and $g[a\times a]\subseteq a$ for all $a\in F$. Then the closure of $a$ under $g$ implies that $s[a]$ is a suborder of $\PPP$. Let $p$ be a condition in $\PPP$. Then there is $\gamma\in a$ such that $\bar{s}(\gamma)$ is a reduct of $p\in\QQQ$ into $\bar{s}[a]$. 
 Pick $\delta\in a$ with $s(\delta)\leq_\PPP s(\gamma)=\bar{s}(f(\gamma))$. Since $s(\gamma)$ is a reduct of $\bar{s}(\gamma)$ into $\PPP$, this shows that the conditions $\bar{s}(\gamma)$ and $s(\delta)=\bar{s}(f(\delta))$ are compatible in $\QQQ$. Then there is $q\in\bar{s}[a]$ with $q\leq_\QQQ \bar{s}(\gamma),s(\delta)$, because $\bar{s}[a]\in\mathrm{Reg}_\kappa(\QQQ)$. By the choice of $\gamma$, this implies that the conditions $p$ and $q$ are compatible in $\QQQ$. Since $\PPP$ is a regular suborder of $\QQQ$, we can conclude that the conditions $p$ and $s(\delta)$ are compatible in both $\QQQ$ and $\PPP$. This shows that $s(\gamma)$ is a reduct of $p$ into $s[a]$. Hence $s[a]$ is a regular suborder of $\PPP$.

 By Proposition \ref{proposition;CharacterizeFLayered} and our assumptions, the above computations show that every regular suborder of $\QQQ$ is $\calF$-layered and this means that $\QQQ$ is completely $\calF$-layered.
\end{proof}

\begin{lemma}\label{lemma:FLayeredCompleteFLayered}
 If $\lambda=\lambda^{{<}\kappa}$ holds, then every completely $\calF$-layered partial order is $\kappa$-stationarily layered. 
\end{lemma}

\begin{proof}
 Let $\QQQ$ be a completely $\calF$-layered partial order and $C$ be a club in $\PP_\kappa(\QQQ)$. By our assumptions, we can use Proposition  \ref{proposition:REgularClosureChainCondition} to construct a regular suborder $\PPP$ of $\QQQ$ of cardinality at most $\lambda$ with the property that $C\cap\PP_\kappa(\PPP)$ is a club in $\PP_\kappa(\PPP)$. An easy computation shows that $\mathrm{Reg}_\kappa(\PPP)\subseteq\mathrm{Reg}_\kappa(\QQQ)$. By definition, $\PPP$ is $\calF$-layered and Lemma \ref{proposition:FlayeredStationaryLayered} shows that $\PPP$ is $\kappa$-stationarily layered. This allows us to conclude that $\emptyset\neq C\cap\mathrm{Reg}_\kappa(\PPP)\subseteq C\cap\mathrm{Reg}_\kappa(\QQQ)$. These computations show that $\mathrm{Reg}_\kappa(\QQQ)$ is stationary in $\PP_\kappa(\QQQ)$.  
\end{proof}

\begin{lemma}\label{lemma:ProductivityCompletelyLayeredPosets}
 Let $\nu<\kappa$ be a cardinal such that $\lambda=\lambda^\nu$, $\mu^\nu<\kappa$ for every $\mu<\kappa$ and $\mathrm{Cl}(\varphi)\in\calF$ for every function $\map{\varphi}{{}^\nu\lambda}{\lambda}$. If $\lambda=\lambda^{{<}\kappa}$ holds, then the class of $\calF$-layered partial orders is closed under $\nu$-support products.  
\end{lemma}

\begin{proof}
 Let $\seq{\QQQ_\delta}{\delta<\rho}$ be a sequence of completely $\calF$-layered partial orders and let $\vec{\QQQ}=\prod_{\delta<\rho}\QQQ_\delta$ denote the corresponding $\nu$-support product.

Let $P$ be a subset of $\vec{\QQQ}$ of cardinality at most $\lambda$. By our assumptions, there is $S\in[\rho]^\lambda$ and a sequence $\seq{\PPP_\delta}{\delta<\rho}$ such that $\PPP_\delta$ is a regular suborder of $\QQQ_\delta$ of cardinality at most $\lambda$ for every $\delta\in S$, $\PPP_\delta$ is the trivial suborder of $\QQQ_\delta$ for every $\delta\in\rho\setminus S$ and $P$ is contained in the corresponding $\nu$-support product $\vec{\PPP}=\prod_{\delta<\rho}\PPP_\delta$. Then $\vec{\PPP}$ is a regular suborder of $\vec{\QQQ}$ and $\vec{\PPP}$ has cardinality at most $\lambda$.

Next, let $\PPP$ be a regular suborder of $\vec{\QQQ}$ of cardinality at most $\lambda$. By the above computations, there is $S\in[\rho]^\lambda$ and a sequence $\seq{\PPP_\delta}{\delta<\rho}$ such that $\PPP_\delta$ is a regular suborder of $\QQQ_\delta$ of cardinality at most $\lambda$ for every $\delta\in S$, $\PPP_\delta$ is the trivial suborder of $\QQQ_\delta$ for every $\delta\in\rho\setminus S$ and $\PPP$ is contained in the corresponding $\nu$-support product $\vec{\PPP}=\prod_{\delta<\rho}\PPP_\delta$. Then $\PPP$ is a regular suborder of $\vec{\PPP}$. Since $\PPP_\delta$ is $\calF$-layered for every $\delta\in S$, our assumptions and Lemma \ref{lemma:FLayeredProducts} show that $\vec{\PPP}$ is $\calF$-layered. By Lemma \ref{lemma:FLayeredCompleteFLayered}, $\vec{\PPP}$ is completely $\calF$-layered and we can conclude that $\PPP$ is $\calF$-layered.  
\end{proof}

\begin{corollary}\label{corollary:ProductsCompletelyFLayered}
 \begin{enumerate}
  \item The class of completely $\calF$-layered partial orders is closed under finite support products.

  \item Let $\nu<\kappa$ be an infinite cardinal with $\lambda=\lambda^\nu$ and $\mu^\nu<\kappa$ for all $\mu<\kappa$. If $\calF$ extends the $\nu$-club filter on $\PP_\kappa(\lambda)$ and $\lambda=\lambda^{{<}\kappa}$ holds, then the class of completely $\calF$-layered partial orders is closed under $\nu$-support products.

  \item If $\kappa$ is a weakly compact cardinal, then the class of completely $\calF_{wc(\kappa)}$-layered partial orders is closed under $\nu$-support products for every $\nu<\kappa$.

  \item If $\calU$ is a normal fine ultrafilter on $\PP_\kappa(\lambda)$ and $\lambda=\lambda^{{<}\kappa}$ holds, then the class of completely $\calU$-layered partial orders is closed under $\nu$-support products for every $\nu<\kappa$.   \qed
 \end{enumerate}
\end{corollary}


\section{Layering at large cardinals}\label{section:FilterLargeCardinals}

In this section, we show that, if a filter $\calF$ on $\PP_\kappa(\lambda)$ is induced by certain large cardinal properties of $\kappa$, then the class of all partial orders of cardinality at most $\lambda$ that satisfy the $\kappa$-chain condition is equal to the class of $\calF$-layered partial orders.

\begin{lemma}\label{lemma:WeaklyCompactFwcLayered}
 If $\kappa$ is weakly compact, then every partial order of cardinality at most $\kappa$ that satisfies the $\kappa$-chain condition is $\calF_{wc(\kappa)}$-layered. 
\end{lemma}

\begin{proof}
 Let $\PPP$ be a partial order of cardinality at most $\kappa$ satisfying the $\kappa$-chain condition and let $\map{s}{\kappa}{\PPP}$ be a surjection. Since $\kappa$ is inaccessible, there is a club $C$ in $\kappa$ such that the following statements hold for all $\alpha,\beta\in C$ with $\alpha<\beta$: 
 \begin{enumerate}
  \item $s[\alpha]$ is a suborder of $\PPP$. 

  \item If $A$ is an antichain in $s[\alpha]$, then there is a maximal antichain $\bar{A}$ in $\PPP$ with $A\subseteq\bar{A}\subseteq s[\beta]$. 
 \end{enumerate}

 By coding the $s$-preimage of the ordering of $\PPP$ and the incompatibility relation of $\PPP$ into a subset of $\kappa$, we find $\langle\Phi,A,a\rangle$ as in Example \ref{example:Relevant Filters}.(\ref{example:item:WCFilter}) such that the corresponding set $R_{\Phi,A,a}$ consists of inaccessible cardinals $\mu\in\Lim(C)$ with the property that the partial order $s[\mu]$ satisfies the $\mu$-chain condition.

 Pick $\mu\in R_{\Phi,A,a}$ and let $A$ be a maximal antichain in $s[\mu]$. Then we can find $\alpha\in\mu\cap C$ with $A\subseteq s[\alpha]$, because $s[\mu]$ satisfies the $\mu$-chain condition. By the definition of $C$, there is a maximal antichain $\bar{A}$ in $\PPP$ with $A\subseteq\bar{A}\subseteq s[\mu]$. Since $s[\mu]$ is a suborder of $\PPP$, we can conclude that $\bar{A}$ is an antichain in $s[\mu]$ and $A=\bar{A}$ is a maximal antichain in $\PPP$. Hence $s[\mu]$ is a regular suborder of $\PPP$. These computations show that $\PPP$ is $\calF_{wc(\kappa)}$-layered.  
\end{proof}

\begin{proof}[Proof of Theorem \ref{theorem:WeaklyCompactEquivalentChainCondition}]
 Let $\kappa$ be a weakly compact cardinal. By Lemma \ref{lemma:StationaryLayeredKnaster} and Lemma \ref{lemma:FLayeredCompleteFLayered}, every completely $\calF_{wc(\kappa)}$-layered partial order is $\kappa$-stationarily layered and $\kappa$-Knaster. Let $\PPP$ be a partial order satisfying the $\kappa$-chain condition. Then Proposition \ref{proposition:REgularClosureChainCondition} implies that every subset of $\PPP$ of cardinality at most $\kappa$ is contained in a regular suborder of $\PPP$ of cardinality at most $\kappa$ and Lemma \ref{lemma:WeaklyCompactFwcLayered} shows that every regular suborder of $\PPP$ of cardinality at most $\kappa$ is $\calF_{wc(\kappa)}$-layered. This shows that $\PPP$ is completely $\calF_{wc(\kappa)}$-layered. 
\end{proof}

A combination of Theorem \ref{theorem:WeaklyCompactEquivalentChainCondition} with Corollary \ref{corollary:ProductsCompletelyFLayered} yields the following statement that provides an alternative proof of Proposition \ref{proposition:WCyieldsCkappa}.

\begin{corollary}
If $\kappa$ is a weakly compact and $\nu<\kappa$, then $\nu$-support products of partial orders satisfying the $\kappa$-chain condition are  $\kappa$-stationarily layered.  \qed
\end{corollary}

In the following, we prove the analogs of the above results for $\lambda$-supercompact cardinals. The following lemma is a consequence of {\cite[Lemma 30]{Cox_Layerings}}. We present a short  direct proof for the reader's convenience.

\begin{lemma}\label{lemma:SupercompactAllPOLambdaULayered}
 If $\kappa$ is a $\lambda$-supercompact cardinal and $\calU$ is a normal fine ultrafilter on $\PP_\kappa(\lambda)$, then every partial order of cardinality at most $\lambda$ that satisfies the $\kappa$-chain condition is $\calU$-layered. 
\end{lemma}

\begin{proof}
 Let $\map{j_\calU}{\VV}{N}$ be the ultrapower embedding corresponding to $\calU$. Fix a partial order $\PPP$ of cardinality at most $\lambda$ satisfying the $\kappa$-chain condition and a surjection $\map{s}{\lambda}{\PPP}$. By elementarity, $j_\calU[\PPP]$ is a suborder of $j_\calU(\PPP)$ and the map $\map{j_\calU\restriction\PPP}{\PPP}{j_\calU[\PPP]}$ is an isomorphism of partial orders. Let $A$ be a maximal antichain in $j_\calU[\PPP]$. Then $A_0=j_\calU^{{-}1}[A]$ is a maximal antichain in $\PPP$ and, since $\PPP$ satisfies the $\kappa$-chain condition, we have $\betrag{A}<\kappa$. By elementarity, this implies that $A=j_\calU[A_0]=j_\calU(A_0)$ is a maximal antichain in $j_\calU(\PPP)$. This argument shows that $j_\calU[\PPP]$ is a regular suborder of $j_\calU(\PPP)$ in $N$. Since $$j_\calU(s)[j_\calU[\lambda]] ~ = ~ j_\calU[s[\lambda]] ~ = ~ j_\calU[\PPP] ~ \in ~ \mathrm{Reg}_{j_\calU(\kappa)}(j_\calU(\PPP))^N ~ = ~ j_\calU(\mathrm{Reg}_\kappa(\PPP)),$$ we can conclude that $\Set{a\in\PP_\kappa(\lambda)}{s[a]\in\mathrm{Reg}_\kappa(\PPP)}\in\calU$. 
\end{proof}

\begin{theorem}\label{theorem:SupercompactULayered}
  If $\kappa$ is a $\lambda$-supercompact cardinal, $\calU$ is a normal fine ultrafilter on $\PP_\kappa(\lambda)$ and $\lambda=\lambda^{{<}\lambda}$ holds, then the following statements are equivalent for every partial order $\PPP$: 
 \begin{enumerate}
  \item $\PPP$ satisfies the $\kappa$-chain condition. 

  \item $\PPP$ is $\kappa$-Knaster. 

  \item $\PPP$ is $\kappa$-stationarily layered. 

  \item $\PPP$ is completely $\calU$-layered. 
 \end{enumerate}
\end{theorem}

\begin{proof}
 By Lemma \ref{lemma:StationaryLayeredKnaster} and Lemma \ref{lemma:FLayeredCompleteFLayered}, every completely $\calU$-layered partial order is $\kappa$-stationarily layered and $\kappa$-Knaster. Let $\PPP$ be a partial order satisfying the $\kappa$-chain condition. Then Proposition \ref{proposition:REgularClosureChainCondition} implies that every subset of $\PPP$ of cardinality at most $\lambda$ is contained in a regular suborder of $\PPP$ of cardinality at most $\lambda$ and Lemma \ref{lemma:SupercompactAllPOLambdaULayered} shows that every regular suborder of $\PPP$ of cardinality at most $\lambda$ is $\calU$-layered. This shows that $\PPP$ is completely $\calU$-layered. 
\end{proof}

In the remainder of this section, we prove Theorem \ref{theorem:KunenModelNonWeaklyCompactKnasterLayered}. This proof uses a construction of Kunen presented in {\cite[Section 3]{MR495118}} showing that it is possible to use forcing to first destroy the weak compactness of a cardinal and then resurrect it by a further forcing. In order to describe this construction, we recall some basic definitions about set-theoretic trees.

Remember that a partial order $\TTT$ is a \emph{tree} if it has a unique minimal element and for every $t\in\TTT$, the set $\mathrm{pred}_\TTT(t)=\Set{s\in\TTT}{s<_\TTT t}$ is well-ordered by $<_\TTT$. In this situation, we set $\length{t}{\TTT}=\otp{\mathrm{pred}_\TTT(t),<_\TTT}$, $\TTT(\alpha)=\Set{t\in\TTT}{\length{t}{\TTT}=\alpha}$ and $\TTT_{{<}\alpha}=\Set{t\in\TTT}{\length{t}{\TTT}<\alpha}$ for every ordinal $\alpha$. We let $\height{\TTT}=\lub_{t\in\TTT}\length{t}{\TTT}$ denote the \emph{height of a tree $\TTT$} and we set $\TTT\restriction A=\bigcup_{\alpha\in A}\TTT(\alpha)$ for every $A\subseteq\height{\TTT}$. A subset $B$ of a tree $\TTT$ is a \emph{cofinal branch through $\TTT$} if $B$ is downwards closed with respect to $\leq_\TTT$ and $B$ is well-ordered by $<_\TTT$ with order-type $\height{\TTT}$. Given an uncountable regular cardinal $\kappa$, a tree $\TTT$ is a \emph{$\kappa$-Aronszajn tree} if $\height{\TTT}=\kappa$, $\TTT$ has no cofinal branches and $\betrag{\TTT(\alpha)}<\kappa$ for all $\alpha<\kappa$. Finally, we say that a tree $\TTT$ of uncountable regular height $\kappa$ is a \emph{$\kappa$-Souslin tree} if the corresponding partial order $\PPP_\TTT=\langle\TTT,\geq_\TTT\rangle$ satisfies the $\kappa$-chain condition.

Kunen's construction starts with a preparatory forcing that causes the weak compactness of a cardinal to be indestructible with respect to forcing with the partial order $\Add{\kappa}{1}$ that adds a Cohen subset to $\kappa$. Then an $\Add{\kappa}{1}$-generic extension is split into a two-step iteration that first adds a $\kappa$-Souslin tree $\TTT$ and then adds a cofinal branch through this tree using the partial order $\PPP_\TTT$. In the intermediate model, $\kappa$ is inaccessible and non-weakly compact; whereas the preparatory forcing ensures that $\kappa$ is again weakly compact in the final extension. Moreover, this construction ensures that the partial order $\PPP_\TTT$ is ${<}\kappa$-distributive in $\VV[G]$.

\begin{proof}[Proof of Theorem \ref{theorem:KunenModelNonWeaklyCompactKnasterLayered}]
 Let $\kappa$ be a weakly compact cardinal. By the results of {\cite[Section 3]{MR495118}} mentioned above, there is a partial order $\PPP$ and a filter $G$ on $\PPP$ that is generic over $\VV$ such that $\kappa$ is inaccessible in $\VV[G]$ and there is a $\kappa$-Souslin tree $\TTT$ in $\VV[G]$ with the property that the corresponding partial order $\PPP_\TTT$ is ${<}\kappa$-distributive in $\VV[G]$ and $\kappa$ is weakly compact in every $\PPP_\TTT$-generic extension of $\VV[G]$. Let $\calF$ denote the set of all subsets $A$ of $\PP_\kappa(\kappa)$ in $\VV[G]$ such that $\mathbbm{1}_{\PPP_\TTT}\Vdash\anf{\check{A}\in\calF_{wc(\check{\kappa})}}$ holds in $\VV[G]$. Then $\calF$ is a normal filter on $\PP_\kappa(\kappa)$ in $\VV[G]$.

 Pick a $\kappa$-Knaster partial order $\QQQ$ of cardinality $\kappa$ and a surjection $\map{s}{\kappa}{\QQQ}$ in $\VV[G]$. Then $\PPP_\TTT\times\QQQ$ satisfies the $\kappa$-chain condition in $\VV[G]$. Define $$F ~ = ~ \Set{a\in\PP_\kappa(\kappa)^{\VV[G]}}{s[a]\in\mathrm{Reg}_\kappa(\QQQ)^{\VV[G]}} ~ \in ~ \PP_\kappa(\kappa)^{\VV[G]}.$$ Let $H$ be $\PPP_\TTT$-generic over $\VV[G]$. In $\VV[G,H]$, $\kappa$ is a weakly compact cardinal and $\QQQ$ satisfies the $\kappa$-chain condition. Since the partial order $\PPP_\TTT$ is ${<}\kappa$-distributive in $\VV[G]$, we know that $\PP_\kappa(\kappa)^{\VV[G]}=\PP_\kappa(\kappa)^{\VV[G,H]}$,  $\PP_\kappa(\QQQ)^{\VV[G]}=\PP_\kappa(\QQQ)^{\VV[G,H]}$ and  $\mathrm{Reg}_\kappa(\QQQ)^{\VV[G]}=\mathrm{Reg}_\kappa(\QQQ)^{\VV[G,H]}$. 
In this situation, Lemma \ref{lemma:WeaklyCompactFwcLayered} implies that $\QQQ$ is $\calF_{wc(\kappa)}^{\VV[G,H]}$-layered in $\VV[G,H]$ and we can conclude that $$F ~ = ~ \Set{a\in\PP_\kappa(\kappa)^{\VV[G,H]}}{s[a]\in\mathrm{Reg}_\kappa(\QQQ)^{\VV[G,H]}} ~ \in ~ \calF_{wc(\kappa)}^{\VV[G,H]}.$$ This argument shows that $F\in\calF$ and $\QQQ$ is $\calF$-layered in $\VV[G]$.

 Now, work in $\VV[G]$ and pick a $\kappa$-Knaster partial order $\QQQ$. By Proposition \ref{proposition:REgularClosureChainCondition}, every subset of $\QQQ$ of cardinality at most $\kappa$ is contained in a regular suborder of $\QQQ$ of cardinality $\kappa$. Since every regular suborder of $\QQQ$ is also $\kappa$-Knaster, the above computations show that every regular suborder of $\QQQ$ of cardinality $\kappa$ is $\calF$-layered.

 These computations show that every $\kappa$-Knaster partial order is completely $\calF$-layered in $\VV[G]$. By Lemma \ref{lemma:FLayeredCompleteFLayered}, this implies that every $\kappa$-Knaster partial order  is $\kappa$-stationarily layered in $\VV[G]$.

Finally, given a cardinal $\nu<\kappa$, a function $\map{\varphi}{{}^\nu\kappa}{\kappa}$ in $\VV[G]$ and a filter $H$ that is $\PPP_\TTT$-generic over $\VV[G]$, we can use the computations made in Example \ref{example:Relevant Filters}.(\ref{example:item:WCFilter}) and the fact that the partial order $\PPP_\TTT$ is ${<}\kappa$-distributive in $\VV[G]$ to conclude that $\mathrm{Cl}(\varphi)^{\VV[G]}=\mathrm{Cl}(\varphi)^{\VV[G,H]}\in\calF_{wc(\kappa)}^{\VV[G,H]}$ holds. This shows that we have $\mathrm{Cl}(\varphi)^{\VV[G]}\in\calF$ for every cardinal $\nu<\kappa$ and every function $\map{\varphi}{{}^\nu\kappa}{\kappa}$ in $\VV[G]$. In this situation, Lemma \ref{lemma:ProductivityCompletelyLayeredPosets} implies that for every $\nu<\kappa$, the class of completely $\calF$-layered partial orders is closed under $\nu$-support products in $\VV[G]$. By combining Lemma \ref{lemma:StationaryLayeredKnaster}, Lemma \ref{lemma:FLayeredCompleteFLayered} and the above computations, we know that the class of $\kappa$-Knaster partial orders coincides with the class of completely $\calF$-layered partial orders in $\VV[G]$. This yields the last statement of the theorem. 
\end{proof}

\begin{remark}
 Proposition \ref{proposition:WCyieldsCkappa} shows that the last statement listed in Theorem \ref{theorem:KunenModelNonWeaklyCompactKnasterLayered} holds when $\kappa$ is weakly compact and the theorem shows that consistently this statement holds at an inaccessible cardinal that is not weakly compact. In contrast, {\cite[Theorem 1.13]{Luecke_Specialize}} shows that this statement characterizes weak compactness in canonical inner models. 
\end{remark}


\section{Characterizations of large cardinals}\label{section:CharacterizeLarge}

This section contains the proofs of the theorems that characterize inaccessible and weakly compact cardinals in terms of layered partial orders.   We start with the characterization of inaccessible cardinals.

\begin{proof}[Proof of Theorem \ref{thm_InaccEquiv}] 
 Let $\kappa$ be an inaccessible cardinal and let $\PPP$ be a ${<}\kappa$-linked partial order. Fix a function $\map{c}{\PPP}{\lambda}$ that is injective on antichains in $\PPP$ and a regular cardinal $\theta>\kappa$ with $\PPP\in\HH{\theta}$. Let $S$ denote the collection of all elementary submodels $M$ of $\HH{\theta}$ such that $\betrag{M}<\kappa$ and $\POT{\lambda}\cup\{\PPP,c\}\subseteq M$. Then $S$ is stationary in $\PP_\kappa(\HH{\theta})$. Fix $M\in S$ and a condition $p$ in $\PPP$. Set $A=\Set{c(r)}{r\leq_\PPP p}\subseteq\lambda$. Then $A\in M$ and elementarity implies that there is a condition $q$ in $\PPP\cap M$ with $A=\Set{c(r)}{r\leq_\PPP q}$.  Let $r$ be an extension of $q$ in $\PPP\cap M$. Then there is an extension $s$ of $p$ in $\PPP$ with $c(r)=c(s)$. This shows that every extension of $q$ in $\PPP\cap M$ is compatible with $p$ in $\PPP$. We can conclude that $q$ is a reduct of $p$ into $\PPP\cap M$. This shows that $\PPP\cap M$ is a regular suborder of $\PPP$.

 In the other direction, assume toward a contradiction that $\kappa$ is regular but not strongly inaccessible and every ${<}\kappa$-linked partial order is $\kappa$-stationarily layered.  Let $\lambda<\kappa$ be minimal with $2^\lambda\geq\kappa$. Since $\kappa$ is regular, we have $2^{{<}\lambda}<\kappa$. Fix a subset $X$ of ${}^\lambda 2$ of cardinality $\kappa$ and define a partial order $\PPP$ by the following clauses. 
  \begin{enumerate}[(a)]
  \item Conditions in $\PPP$  are pairs $p=\langle s_p,a_p\rangle$ with $\map{s_p}{n_p}{{}^{{<}\lambda}2}$ for some $n_p<\omega$ and $a_p\in\left[X\right]^{{<}\omega}$. 

  \item Given conditions $p$ and $q$ in $\PPP$, we have $p\leq_\PPP q$ if and only if $s_q\subseteq s_p$, $a_q\subseteq a_p$ and $s_p(i)\nsubseteq x$ for all $x\in a_q$ and $n_q\leq i<n_p$.
 \end{enumerate}
 Since conditions in $\PPP$ with the same first component are compatible, the partial order $\PPP$ is ${<}\kappa$-linked and has cardinality $\kappa$. Let $\theta>\kappa$ be a regular cardinal with $\PPP\in\HH{\theta}$. By Lemma \ref{lemma:StationaryLayeredElementarySubmodel}, there is an elementary submodel $M$ of $\HH{\theta}$ of cardinality less than $\kappa$ such that $\PPP\in M$, ${}^{{<}\lambda}2\subseteq M$ and $\PPP\cap M$ is a regular suborder of $\PPP$. Pick $x\in X\setminus M$ and define $p=\langle\emptyset,\{x\}\rangle$. Then $p$ is a condition in $\PPP$ and there is reduct $q$ of $p$ into $\PPP\cap M$. Since $x \notin M$ and $a_q$ is a finite subset of $M$, we have $x\notin a_q$ and there is an $\alpha<\lambda$ with $x\restriction\alpha\neq y\restriction\alpha$ for every $y\in a_q$.   Define $r=\langle s_q\cup\{\langle n_q,x\restriction\alpha\rangle\},a_q\rangle$.  Notice that $x \restriction \alpha$ is an element of $M$ because ${}^{{<}\lambda}2\subseteq M$.  Then $r$ is an extension of $q$ in $\PPP\cap M$ and our assumptions imply that $r$ is compatible with $p$ in $\PPP$, a contradiction. 
\end{proof}

\begin{remark}
Essentially the same proof shows that the partial order $\PPP$ defined in the previous proof is not \emph{strongly proper} with respect to any stationary subset of $\PP_\kappa(\HH{\theta})$.  Here \emph{strongly proper} is in the sense of Mitchell (see \cite{MR2279659}).   
\end{remark}

To characterize weakly compact cardinals, we make use of the following well-known forcing notion that turns a tree of uncountable regular height without cofinal branches into a special tree of height $\omega_1$.

\begin{definition}\label{def_SpecializingForcing}
 Given an uncountable regular cardinal $\kappa$ and a tree $\TTT$ of height $\kappa$, we let $\PPP(\TTT)$ denote the partial order whose conditions are finite partial functions $\pmap{s}{\TTT}{\omega}{\textit{finite}}$ that are injective on chains in $\TTT$ and whose ordering is given by reversed inclusion. 
\end{definition}

The proof of the following lemma is a small variation of Baumgartner's elegant proof of {\cite[Theorem 8.2]{MR823775}}. A detailed presentation of this argument can be found in \cite{Luecke_Specialize}.

\begin{lemma}\label{lemma:Baumgartner}
 If $\kappa$ is an uncountable regular cardinal and $\TTT$ is a tree of height $\kappa$ without cofinal branches, then the partial order $\PPP(\TTT)$ satisfies the $\kappa$-chain condition. 
\end{lemma}

The next lemma is the last ingredient missing in order to complete the proof of Theorem \ref{thm_WeakCompactEquiv}.

\begin{lemma}\label{lemma:SpecialAronszajnNotLayered}
 If $\kappa$ is an uncountable regular cardinal and $\TTT$ is a $\kappa$-Aronszajn tree, then the partial order $\PPP(\TTT)$ is not $\kappa$-stationarily layered. 
\end{lemma}

\begin{proof}
 Assume toward a contradiction that $\PPP(\TTT)$ is $\kappa$-stationarily layered. By Lemma \ref{lemma:StationaryLayeredElementarySubmodel}, there is an elementary submodel $M$ of $\HH{\kappa^+}$ such that $\betrag{M}<\kappa$, $\kappa\cap M\in\kappa$, $\PPP\in M$ and $\PPP(\TTT)\cap M\in\mathrm{Reg}_\kappa(\PPP(\TTT))$. Pick $t\in\TTT(\kappa\cap M)$ and set $p=\{\langle t,0\rangle\}$. Then $p$ is a condition in $\PPP(\TTT)$ and there is a reduct $q$ of $p$ into $\PPP(\TTT)\cap M$. Since the conditions $p$ and $q$ are compatible in $\PPP(\TTT)$, we have $q(s)\neq 0$ for all $s\in\dom{q}$ with $s<_\TTT t$. Let $\beta<\kappa$ be minimal with $\dom{q}\subseteq\TTT_{{<}\beta}$. Then $\beta<\kappa\cap M$ and elementarity implies that $\TTT(\beta)\subseteq M$, because $\TTT$ is a $\kappa$-Aronszajn tree. Let $u$ denote the unique element of $\TTT(\beta)$ with $u<_\TTT t$.  Set $r=q\cup\{\langle u,0\rangle\}$. By the above remarks, $r$ is a condition in $\PPP(\TTT)\cap M$ below $q$. This implies that the conditions $p$ and $r$ are compatible in $\PPP(\TTT)$, a contradiction. 
\end{proof}

\begin{remark}
Essentially the same proof shows that the specializing partial order $\PPP(\TTT)$ is not strongly proper with respect to any stationary subset of $\PP_\kappa(\HH{\kappa^+})$.
\end{remark}

\begin{proof}[Proof of Theorem \ref{thm_WeakCompactEquiv}]
 By Theorem \ref{theorem:WeaklyCompactEquivalentChainCondition}, if $\kappa$ is a weakly compact cardinal, then every partial order satisfying the $\kappa$-chain condition is $\kappa$-stationarily layered.

 In the converse direction, let $\kappa$ be an uncountable regular cardinal with the property that every partial order of cardinality at most $\kappa$ satisfying the $\kappa$-chain condition is $\kappa$-stationarily layered. Since ${<}\kappa$-linked partial orders satisfy the $\kappa$-chain condition, Theorem \ref{thm_InaccEquiv} implies that $\kappa$ is inaccessible.  In this situation,  a combination of Lemma \ref{lemma:Baumgartner} and Lemma \ref{lemma:SpecialAronszajnNotLayered} implies that there are no $\kappa$-Aronszajn trees. By a classical result of Erd{\H{o}}s and Tarski, we can conclude that $\kappa$ is weakly compact. 
\end{proof}


\section{Characterizations via the Knaster property}\label{section:Knaster}

In this section, we prove Theorem \ref{theorem:KnasterStationary}. The proof of this  result makes use of the following notions defined by Todor{\v{c}}evi{\'c} in {\cite[Section 4]{MR793235}} and {\cite[Section 1]{MR908147}}.

\begin{definition}[Todor{\v{c}}evi{\'c}]
 Let $\kappa$ be an uncountable regular cardinal, let $S\subseteq\kappa$ and let $\TTT$ be a tree of height $\kappa$. 
 \begin{enumerate}
   \item A map $\map{r}{\TTT\restriction S}{\TTT}$ is \emph{regressive} if $r(t)<_\TTT t$ holds for $t\in\TTT\restriction S$ that is not minimal in $\TTT$. 

  \item We say that \emph{$S$ is nonstationary with respect to $\TTT$} if there is a regressive map $\map{r}{\TTT\restriction S}{\TTT}$ with the property that for every $t\in\TTT$ there is a function $\map{c_t}{r^{{-}1}\{t\}}{\lambda_t}$ such that $\lambda_t<\kappa$ and $c_t$ is injective on $\leq_\TTT$-chains. 

   \item The tree $\TTT$ is \emph{special} if $\kappa$ is nonstationary with respect to $\TTT$. 
 \end{enumerate}
\end{definition}

By {\cite[Theorem 14]{MR793235}}, this notion of special trees generalizes the classical notion of \emph{special $\kappa^+$-trees} (trees of height $\kappa^+$ that are a union of $\kappa$-many antichains) to a concept that also makes sense for limit cardinals. Todor{\v{c}}evi{\'c} used this concept in {\cite[Theorem 1.9]{MR908147}} to provide a new characterization of Mahlo cardinals. His result shows that an inaccessible cardinal $\kappa$ is a Mahlo cardinal if and only if there are no special $\kappa$-Aronszajn trees.

\begin{proposition}\label{proposition:SSpecialNoCofinalBranches}
 Let $\kappa$ be an uncountable regular cardinal and let $\TTT$ be a tree of height $\kappa$. If there is a stationary subset $S$ of $\kappa$ that is nonstationary with respect to $\TTT$, then $\TTT$ has no cofinal branches. 
\end{proposition}

\begin{proof}
 Let $\map{r}{\TTT\restriction S}{\TTT}$ and $\seq{\map{c_t}{r^{{-}1}\{t\}}{\lambda_t}}{t\in\TTT}$ witness that $S$ is nonstationary with respect to $\TTT$. Assume toward a contradiction that there is a cofinal branch $B$ through $\TTT$. Given $\alpha\in S$, let $t_\alpha$ denote the unique element of $B\cap\TTT(\alpha)$. By Fodor's Lemma, there is $E\subseteq S$ stationary in $\kappa$, $s\in\TTT$ and $\xi<\lambda_s$ such that $r(t_\alpha)=s$ and $c_s(t_\alpha)=\xi$ for all $\alpha\in E$. This is a contradiction, because $c_s$ is injective on $\Set{t_\alpha}{\alpha\in E}$. 
\end{proof}

 We now want to isolate properties of trees $\TTT$ that imply that the partial order $\PPP(\TTT)$ defined in Definition \ref{def_SpecializingForcing} is $\kappa$-Knaster.

\begin{definition}
  We say that a tree $\TTT$ of height $\nu$ \emph{does not split at limit levels} if for all $\lambda\in\nu\cap\Lim$ and all $t_0,t_1\in\TTT(\lambda)$ with $t_0\neq t_1$, we can find $\alpha<\lambda$ and $s_0,s_1\in\TTT(\alpha)$ with $s_0\neq s_1$ and $s_i<_\TTT t_i$ for all $i<2$.  
\end{definition}

The next lemma is a special case of results contained in  \cite{Luecke_Specialize} that study the properties of forcings that specialize higher Aronszajn trees.

\begin{lemma}\label{lemma:SpecialTreeKnaster}
 Let $\kappa$ be an uncountable regular cardinal and let $\TTT$ be a $\kappa$-Aronszajn tree that does not split at limit levels. If there is a stationary subset of $\kappa$ that is nonstationary with respect to $\TTT$, then the partial order $\PPP(\TTT)$  is $\kappa$-Knaster. 
\end{lemma}

\begin{proof}
 Fix an injective sequence $\seq{p_\alpha}{\alpha<\kappa}$ of conditions in $\PPP(\TTT)$ and functions $\map{r}{\TTT\restriction S}{\TTT}$ and $\seq{\map{c_t}{r^{{-}1}\{t\}}{\lambda_t}}{t\in\TTT}$ witnessing that a stationary subset $S$ of $\kappa$ is nonstationary with respect to $\TTT$.

 Pick $\alpha\in S\cap\Lim$. Then we can find $n_\alpha<\omega$ and a bijective enumeration $\seq{t^\alpha_k}{k<n_\alpha}$ of all  $t\in\TTT(\alpha)$ with $t\leq_\TTT u$ for some $u\in\dom{p_\alpha}$. Since $\TTT$ does not split at limit levels, we can find $\bar{\alpha}<\alpha$ and an injection $\map{\iota_\alpha}{n_\alpha}{\TTT(\bar{\alpha})}$ such that $\iota_\alpha(k)<_\TTT t^\alpha_k$ for all $k<n_\alpha$. Since $\alpha\in\Lim$, there is $\rho_\alpha<\alpha$ such that $\dom{p_\alpha}\cap\TTT_{{<}\alpha}\subseteq\TTT_{{<}\rho_\alpha}$ and  $\{\iota_\alpha(k),r(t^\alpha_k)\}\subseteq\TTT_{{<}\rho_\alpha}$ for all $k<n_\alpha$.

Since $S\cap\Lim$ is stationary in $\kappa$, Fodor's Lemma yields $S_0\subseteq S\cap\Lim$ stationary in $\kappa$, $n_*<\omega$ and $\rho_*<\min(S_0)$ with $n_*=n_\alpha$ and $\rho_*=\rho_\alpha$ for all $\alpha\in S_0$. We have $\betrag{\TTT_{{<}\rho_*}}<\kappa$, because $\TTT$ is a $\kappa$-Aronszajn tree. Another application of Fodor's Lemma yields $S_1\subseteq S_0$ stationary in $\kappa$, $R\subseteq \TTT_{{<}\rho_*}$, a map $\map{d}{R}{\omega}$, an injection $\map{\iota}{n_*}{\TTT}$ and a map $\map{\bar{r}}{n_*}{\TTT}$ such that $R=\dom{p_\alpha}\cap\TTT_{{<}\alpha}$, $d=p_\alpha\restriction R$, $\iota(k)=\iota_\alpha(k)$ and $\bar{r}(k)=r(t^\alpha_k)$ for all $\alpha\in S_1$ and $k<n_*$. Since $\lambda_{\bar{r}(k)}<\kappa$ for all $k<n_*$, a final application of  Fodor's Lemma yields $S_2\subseteq S_1$ stationary in $\kappa$ and a map $\map{c}{n_*}{\kappa}$ such that $c(k)=c_{\bar{r}(k)}(t^\alpha_k)$ for all $\alpha\in S_2$ and $k<n_*$. Pick $U\subseteq S_2$ unbounded in $\kappa$ such that $\dom{p_{\alpha_0}}\subseteq\TTT_{{<}\alpha_1}$ holds for all $\alpha_0,\alpha_1\in U$ with $\alpha_0<\alpha_1$.

 Assume toward a contradiction that there are $\alpha_0,\alpha_1\in U$ such that $\alpha_0<\alpha_1$ and the conditions $p_{\alpha_0}$ and $p_{\alpha_1}$ are incompatible in $\PPP(\TTT)$. Since the above choices ensure that  $\dom{p_{\alpha_0}}\cap\dom{p_{\alpha_1}}=R$ and $p_{\alpha_0}\restriction R=p_{\alpha_1}\restriction R$ hold, we can find $u_0\in\dom{p_{\alpha_0}}\setminus\TTT_{{<}\alpha_0}$ and $u_1\in\dom{p_{\alpha_1}}\setminus\TTT_{{<}\alpha_1}$ with the property that $u_0<_\TTT u_1$ and $p_{\alpha_0}(u_0)=p_{\alpha_1}(u_1)$. Given $i<2$, there is $k_i<n_*$ with $t^{\alpha_i}_{k_i}\leq_\TTT u_i$. Since $\iota(k_0)<_\TTT t^{\alpha_0}_{k_0}\leq_\TTT u_0<_\TTT u_1$, $\iota(k_1)<_\TTT t^{\alpha_1}_{k_1}\leq_\TTT u_1$ and $\ran{\iota}\subseteq\TTT(\beta)$ for some $\beta<\kappa$, the injectivity of $\iota$ implies that $k_0=k_1$. This shows that the nodes $t^{\alpha_0}_{k_0}$ and $t^{\alpha_1}_{k_1}$ are incompatible in $\TTT$, because we have $r(t^{\alpha_0}_{k_0})=\bar{r}(k_0)=\bar{r}(k_1)=r(t^{\alpha_1}_{k_1})$ and $$c_{\bar{r}(k_0)}(t^{\alpha_0}_{k_0}) ~ = ~ c(k_0) ~ = ~ c(k_1) ~ = ~ c_{\bar{r}(k_1)}(t^{\alpha_1}_{k_1}) ~ = ~ c_{\bar{r}(k_0)}(t^{\alpha_1}_{k_1}).$$ But this yields a contradiction, because $t^{\alpha_0}_{k_0}\leq_\TTT u_0<_\TTT u_1$ and $t^{\alpha_1}_{k_1}\leq_\TTT u_1$.

 The above argument shows that the sequence $\seq{p_\alpha}{\alpha\in U}$ consists of pairwise compatible conditions in $\PPP(\TTT)$. 
\end{proof}

In order to prove Theorem \ref{theorem:KnasterStationary}, we introduce more notions defined by Todor{\v{c}}evi{\'c} in {\cite[Section 1]{MR908147}}. These concepts will allow us to combine the above results with the proof of {\cite[Theorem 1.9]{MR908147}} to derive the statement of the theorem.

\begin{definition}[Todor{\v{c}}evi{\'c}]
 Let $\kappa$ be an uncountable regular cardinal.
 \begin{enumerate}
  \item A sequence $\vec{C}=\seq{C_\alpha}{\alpha<\kappa}$ is a \emph{C-sequence of length $\kappa$} if the following statements hold for all $\alpha<\kappa$. 
   \begin{enumerate}
     \item If $\alpha$ is a limit ordinal, then $C_\alpha$ is a closed unbounded subset of $\alpha$. 

     \item If $\alpha=\bar{\alpha}+1$, then $C_\alpha=\{\bar{\alpha}\}$. 
   \end{enumerate}

  \item Let $\vec{C}=\seq{C_\alpha}{\alpha<\kappa}$ be a $C$-sequence. 
   \begin{enumerate}
    \item Given $\alpha\leq\beta<\kappa$, the \emph{walk from $\beta$ to $\alpha$ through $\vec{C}$} is the unique sequence $\langle\gamma_0,\ldots,\gamma_n\rangle$ with  $\gamma_0=\beta$, $\gamma_n=\alpha$ and $\gamma_{i+1}=\min(C_{\gamma_i}\setminus\alpha)$ for all $i<n$. In this situation, we define the \emph{full code of the walk from $\beta$ to $\alpha$ through $\vec{C}$} to be the sequence $$\rho_0^{\vec{C}}(\alpha,\beta) ~ = ~ \langle \otp{C_{\gamma_0}\cap\alpha},\ldots,\otp{C_{\gamma_{n-1}}\cap\alpha}\rangle.$$ 

    \item Given $\beta<\kappa$, we define $$\Map{\rho_0^{\vec{C}}( ~ \cdot ~ ,\beta)}{\beta+1}{{}^{{<}\omega}\kappa }{\alpha}{\rho_0^{\vec{C}}(\alpha,\beta)}.$$ 

     \item We define $\TTT(\rho_0^{\vec{C}})$ to be the tree of height $\kappa$ consisting of all functions of the form $\rho_0^{\vec{C}}( ~ \cdot ~ ,\beta)\restriction\alpha$ with $\alpha\leq\beta<\kappa$ ordered by inclusion. 
   \end{enumerate}
 \end{enumerate}
\end{definition}

Note that trees of the form $\TTT(\rho_0^{\vec{C}})$ for some $C$-sequence $\vec{C}$ do not split at limit levels, because we have $$\length{\rho_0^{\vec{C}}( ~ \cdot ~ ,\beta)\restriction\alpha}{\TTT(\rho_0^{\vec{C}})} ~ = ~ \alpha$$ for all $\alpha\leq\beta<\kappa$.

\begin{proof}[Proof of Part (i) of Theorem \ref{theorem:KnasterStationary}]
 Let $\kappa$ be an uncountable regular cardinal with the property that every $\kappa$-Knaster partial order is $\kappa$-stationarily layered. Since ${<}\kappa$-linked partial orders are $\kappa$-Knaster, Theorem \ref{thm_InaccEquiv} implies that $\kappa$ is inaccessible.  Assume, towards a contradiction, that $\kappa$ is not Mahlo. Then the proof of {\cite[Theorem 1.9]{MR908147}} shows that there is a $C$-sequence $\vec{C}$ of length $\kappa$ such that $\TTT(\rho_0^{\vec{C}})$ is a special $\kappa$-Aronszajn tree. By the above remarks, this implies that there is a special $\kappa$-Aronszajn tree $\TTT$ that does not split at limit levels. Then Lemma \ref{lemma:SpecialTreeKnaster} shows that the partial order $\PPP(\TTT)$ is $\kappa$-Knaster and hence our assumption implies that $\PPP(\TTT)$ is $\kappa$-stationarily layered. But this conclusion contradicts Lemma \ref{lemma:SpecialAronszajnNotLayered}.  
\end{proof}

In the remainder of this section, we prove the second part of Theorem \ref{theorem:KnasterStationary}. We want to construct $\kappa$-Aronszajn trees $\TTT$ with the property that the corresponding partial order $\PPP(\TTT)$ defined in Definition \ref{def_SpecializingForcing} is $\kappa$-Knaster for every inaccessible cardinal $\kappa$ that admits a non-reflecting stationary subset. The following concept will provide us with a method to construct such trees that can also be applied in the case where $\kappa$ is a Mahlo cardinal.

\begin{definition}[{\cite[Definition 8.2.3]{MR2355670}}]
 Given an uncountable regular cardinal $\kappa$, we say that a $C$-sequence $\vec{C}=\seq{C_\alpha}{\alpha<\kappa}$ \emph{avoids $S\subseteq\kappa$} if $C_\alpha\cap S=\emptyset$ holds for every limit ordinal $\alpha$ less than $\kappa$. 
\end{definition}

The proof of the following lemma is similar to the proof of {\cite[Theorem 1.9]{MR908147}}. It relies on several computations made in \cite{MR908147} that deal with the properties of full codes of walks through $C$-sequences. These computations are presented in detail in {\cite[Section 3]{MR3078820}} and we will refer to this presentation in the following proof.

\begin{lemma}\label{lemma:NonstationaryInWalksTree}
 Given an uncountable regular cardinal $\kappa$, if $\vec{C}$ is a $C$-sequence that avoids a subset $S$ of $\kappa$, then $S$ is nonstationary with respect to $\TTT(\rho_0^{\vec{C}})$.  
\end{lemma}

\begin{proof}
 Fix a bijection $\map{b}{\kappa}{{}^{{<}\omega}\kappa}$ and let $D$ be the club of all limit ordinals $\alpha<\kappa$ with $b[\alpha]={}^{{<}\omega}\alpha$. Let $\vec{C}=\seq{C_\alpha}{\alpha<\kappa}$ be a $C$-sequence that avoids a subset $S$ of $\kappa$ and let $\TTT$ denote $\TTT(\rho_0^{\vec{C}})$. Since {\cite[Theorem 13]{MR793235}} shows that the collection of subsets of $\kappa$ that are nonstationary with respect to $\TTT$ is a normal ideal on $\kappa$, it suffices to show that the set $D\cap S$ is nonstationary with respect to $\TTT$ to derive the statement of the lemma.

 Pick $\alpha\in D\cap S$ and $\alpha\leq\beta<\kappa$. Let $\langle\gamma_0,\ldots,\gamma_n\rangle$ denote the walk from $\beta$ to $\alpha$ through $\vec{C}$. Since $\vec{C}$ avoids $S$ and $\alpha\in S$ is a limit ordinal, we can conclude that for all $i<n$, the set  $C_{\gamma_i}\cap\alpha$ is bounded in $\alpha$ and $\otp{C_{\gamma_i}\cap\alpha}<\alpha$ holds. This shows that $\rho_0^{\vec{C}}(\alpha,\beta)\in{}^{{<}\omega}\alpha$ and, by our assumptions on $\alpha$, there is $f(\alpha,\beta)<\alpha$ with $b(f(\alpha,\beta))=\rho_0^{\vec{C}}(\alpha,\beta)$.

Now, pick a node $t$ in $\TTT\restriction(D\cap S)$. Let $\alpha=\length{t}{\TTT}\in D\cap S$ and let $\alpha\leq\beta<\kappa$ be minimal with $t=\rho_0^{\vec{C}}( ~ \cdot ~ ,\beta)\restriction\alpha$. In this situation, we define $$r(t) ~ = ~ \rho_0^{\vec{C}}( ~ \cdot ~ ,\beta)\restriction f(\alpha,\beta) ~ <_{\TTT} ~ t.$$ This yields a regressive function $\map{r}{\TTT\restriction(D\cap S)}{\TTT}$.  Assume toward a contradiction that there are $t_0,t_1\in\TTT\restriction(D\cap S)$ with $t_0<_\TTT t_1$ and $r(t_0)=r(t_1)$. Given $i<2$, let $\alpha_i=\length{t_i}{\TTT}$ and let $\alpha_i\leq\beta_i<\kappa$ be minimal with $t_i=\rho_0^{\vec{C}}( ~ \cdot ~ ,\beta_i)\restriction\alpha_i$. Then $\alpha_0<\alpha_1$, $\rho_0^{\vec{C}}(\alpha_0,\beta_0)=\rho_0^{\vec{C}}(\alpha_1,\beta_1)$ and $\rho_0^{\vec{C}}(\xi,\beta_0)=\rho_0^{\vec{C}}(\xi,\beta_1)$ for all $\xi<\alpha_0$.

First, assume that $\alpha_0=\beta_0$. Then $\emptyset=\rho_0^{\vec{C}}(\alpha_0,\beta_0)=\rho_0^{\vec{C}}(\alpha_1,\beta_1)$ and we have $\alpha_1=\beta_1$. Since $\vec{C}$ avoids $S$, we have $\max(C_{\alpha_1}\cap\alpha_0)<\alpha_0$ and there is $\xi\in C_{\alpha_0}$ with $\max(C_{\alpha_1}\cap\alpha_0)<\xi<\alpha_0$. By our assumptions, we have $$\langle\otp{C_{\alpha_0}\cap\xi}\rangle ~ = ~ \rho_0^{\vec{C}}(\xi,\alpha_0) ~ = ~ \rho_0^{\vec{C}}(\xi,\alpha_1).$$ and the sequence $\langle\alpha_1,\xi\rangle$ is the walk from $\alpha_1$ to $\xi$ through $\vec{C}$. But this implies that $\xi\in C_{\alpha_1}$, a contradiction.

 Now, assume that $\alpha_0<\beta_0$. Given $i<2$, let $\langle\gamma^i_0,\ldots,\gamma^i_{n_i}\rangle$ denote the walk from $\beta_i$ to $\alpha_i$ through $\vec{C}$. Then $\rho_0^{\vec{C}}(\alpha_0,\beta_0)=\rho_0^{\vec{C}}(\alpha_1,\beta_1)$ implies that $n_0=n_1>0$ and $\alpha_1<\beta_1$. Since $\vec{C}$ avoids $S$, we have  $\sup(C_{\gamma^i_{n_i-1}}\cap \alpha_0)<\alpha_0$ for all $i<2$. Together with our assumptions, this allows us to use {\cite[Proposition 3.6]{MR3078820}} to conclude that $$\rho_0^{\vec{C}}(\alpha_0,\beta_0) ~ = ~ \rho_0^{\vec{C}}(\alpha_0,\beta_1)$$ holds. Moreover, since $\alpha_0<\alpha_1<\beta_1$, we can apply {\cite[Lemma 3.4]{MR3078820}} to see that $$\rho_0^{\vec{C}}(\alpha_0,\beta_1) ~ \neq ~ \rho_0^{\vec{C}}(\alpha_1,\beta_1)$$ holds. Together with our assumptions, this yields $$\rho_0^{\vec{C}}(\alpha_0,\beta_0) ~ = ~ \rho_0^{\vec{C}}(\alpha_0,\beta_1)~ \neq ~ \rho_0^{\vec{C}}(\alpha_1,\beta_1) ~ = ~ \rho_0^{\vec{C}}(\alpha_0,\beta_0),$$ a contradiction.

By the above computations, the regressive function $\map{r}{\TTT\restriction(D\cap S)}{\TTT}$ witnesses that the set $D\cap S$ is nonstationary with respect to $\TTT$. 
\end{proof}

By the above lemma, it is possible to construct trees with the desired  properties from a $C$-sequences $\vec{C}$ with the properties that $\TTT(\rho_0^{\vec{C}})$ is a $\kappa$-Aronszajn tree and $\vec{C}$ avoids a stationary subset of $\kappa$. The following proof shows that the sequence induced by a non-reflecting subset has these properties.

\begin{proof}[Proof of Part (ii) of Theorem \ref{theorem:KnasterStationary}]
 Let $\kappa$ be an uncountable regular cardinal with the property that every $\kappa$-Knaster partial order of cardinality at most $\kappa$ is $\kappa$-stationarily layered. Then Theorem \ref{thm_InaccEquiv} implies that $\kappa$ is inaccessible.  Assume  toward a contradiction that there is a stationary subset $S$ of $\kappa$ that does not reflect. Then we may assume that $S$ consists of limit ordinals and, by picking counterexamples to reflection at every limit ordinal less than $\kappa$, we can construct a $C$-sequence $\vec{C}=\seq{C_\alpha}{\alpha<\kappa}$ that avoids $S$. Set $\TTT=\TTT(\rho_0^{\vec{C}})$. Then Lemma \ref{lemma:NonstationaryInWalksTree} shows that there is a stationary subset of $\kappa$ that is nonstationary with respect to $\TTT$. By Proposition \ref{proposition:SSpecialNoCofinalBranches}, this implies that $\TTT$ has no cofinal branches. Since $\kappa$ is inaccessible, we can use {\cite[Proposition 4.2]{MR3078820}} to conclude that $\betrag{\TTT(\alpha)}<\kappa$ holds for all $\alpha<\kappa$. This shows that $\TTT$ is a $\kappa$-Aronszajn tree. By the above computations, Lemma \ref{lemma:SpecialTreeKnaster} implies that the partial order $\PPP(\TTT)$ is $\kappa$-Knaster. But Lemma \ref{lemma:SpecialAronszajnNotLayered} shows that $\PPP(\TTT)$ is not $\kappa$-stationarily layered, a contradiction.  
\end{proof}


\section{Questions and concluding remarks}\label{sec_Questions}

We conclude with questions raised by the results of this paper.

\begin{question}
 Assume that $\kappa$ is an inaccessible cardinal with the property that every $\kappa$-Knaster partial order is $\kappa$-stationarily layered. Is $\kappa$ weakly compact in $\LL$?  
\end{question}

Theorem \ref{theorem:KnasterStationary} shows that cardinals with this property are \emph{reflection cardinals} and, by {\cite[Theorem 3]{MR1029909}}, they are reflection cardinals in $\LL$. Moreover,  Corollary \ref{corollary:KnasterLayeredExtenderModels} shows that in $\LL$, the above assumption is equivalent to weak compactness. In contrast, Theorem \ref{theorem:KunenModelNonWeaklyCompactKnasterLayered} shows that the assumption does not imply weak compactness in $\VV$. The following question mentions the obvious strategy to derive a positive answer to the above question.

\begin{question}\label{q_KnasterSquare}
 Let $\kappa$ be an inaccessible cardinal. Does the existence of a  $\square(\kappa)$-sequence (see {\cite[Section 1]{MR908147}}) imply the existence of a $\kappa$-Knaster partial order that is not $\kappa$-stationarily layered? 
\end{question}

The results of this paper show that inaccessiblity and weak compactness are provably equivalent to statements about the stationary layeredness of certain partial orders. It is natural to ask whether other small large cardinal properties can be described in this way.

\begin{question}
Are there other natural instances of pairs $(\Phi,\Gamma)$ with 
\begin{itemize}
 \item $\Phi(\kappa)$ is a large cardinal property weaker than weak compactness of $\kappa$, and 

 \item $\Gamma(\kappa)$ is a class of partial orders satisfying the $\kappa$-chain condition. 
\end{itemize}
so that $\ZFC$ proves that for every inaccessible cardinal $\kappa$, the statement $\Phi(\kappa)$ is equivalent to the statement that every partial order in $\Gamma(\kappa)$ is $\kappa$-stationarily layered.  

In particular, is there a class of partial orders satisfying the $\kappa$-chain condition that corresponds to Mahlo cardinals in this way? 
\end{question}

The above questions deal with characterizations of consequences of weak compactness. In the light of Theorem \ref{theorem:SupercompactULayered}, it is natural to ask whether stronger large cardinal properties can be described by similar results.

\begin{question}\label{q_IsKappaMeasurable}
 Let $\kappa$ be an inaccessible cardinal such that there is a normal filter $\calF$ on $\PP_\kappa(\kappa^+)$ with the property that every partial order of cardinality $\kappa^+$ that satisfies the $\kappa$-chain condition is $\calF$-layered.  Must $\kappa$ be a measurable cardinal?
\end{question}

In the setting of the above question, we have $\kappa^+=(\kappa^+)^{{<}\kappa}$ and, by Proposition \ref{proposition:REgularClosureChainCondition} and Lemma \ref{lemma:FLayeredCompleteFLayered}, every partial order satisfying the $\kappa$-chain condition is completely $\calF$-layered and therefore $\kappa$-stationarily layered. In particular, Theorem \ref{thm_WeakCompactEquiv} shows that the above assumptions on $\kappa$ imply that $\kappa$ is weakly compact.

The characterization of weakly compact and Mahlo cardinals through the non-existence of certain trees is crucial for the above proofs. Since stronger large cardinal properties can be described by the non-existence of certain $\PP_\kappa(\lambda)$-trees (see \cite{MR0325397} and \cite{MR0327518}), we may ask whether similar arguments could work for these properties. This approach leads to the question whether analogs of the specialization forcing $\PPP(\TTT)$ exist for $\PP_\kappa(\lambda)$-trees.

\begin{question}
 Given an inaccessible cardinal $\kappa$, a cardinal $\lambda>\kappa$ and a normal filter $\calF$ on $\PP_\kappa(\lambda)$, 
 does the existence of a $\PP_\kappa(\lambda)$-tree without cofinal branches imply the existence of a partial order of cardinality $\lambda$ that satisfies the $\kappa$-chain condition and is not $\calF$-layered? 
\end{question}

\begin{figure}[ht]\label{fig_Implications}
 \begin{displaymath}
  \xymatrix{
     & \kappa\text{-c.c. is productive} \ar@{-->}[dl]|{ \textcircled{\tiny ?} }   \ar[dr] & \\
    \txt{$\kappa$ is weakly \\ compact} \ar@{<->}[r] & \exists \mathcal{F} ~ \left[\text{$\kappa$-c.c.} \subseteq \text{$\mathcal{F}$-layered}\right] \ar[u] \ar@{<->}[d] & \neg \square(\kappa) \ar[d] \\ 
    & \txt{$\kappa$-c.c. \\ $\subseteq$ $\kappa$-stat.-layered} \ar[d] & \txt{$\kappa$ is weakly  \\ compact in $\LL$} \\ 
   \txt{$\kappa$ is Mahlo and \\ stat. sets reflect} \ar@{<->}[uu]|{\text{If $\VV=\LL[E]$}} \ar[d] &  \txt{$\kappa$-Knaster \\ $\subseteq$ $\kappa$-stat.-layered} \ar@{-->}[uur]|{\textcircled{\tiny ?}} \ar[d] \ar[l]  \ar[uul]|{\textcircled{\small \sf{x}}}  \ar@{-->}[ur]|{\textcircled{\tiny ?}} & \\
\txt{$\kappa$ is strongly \\ inaccessible} \ar@{<->}[r] & \txt{${<}\kappa$-linked \\ $\subseteq$ $\kappa$-stat.-layered} & \\
 }
\end{displaymath}
\caption{Summary of results and open questions.}
\label{figure:ResultsQuestions}
\end{figure}

Figure \ref{figure:ResultsQuestions} summarizes some of the main theorems and open questions related to the topics of this paper.  An arrow in the diagram represents an implication; a crossed-out arrow represents a non-implication; and a dashed arrow with a question mark represents an open problem.  The motivating problem -- whether productivity of the $\kappa$-chain condition implies the weak compactness of $\kappa$ -- appeared in Todor{\v{c}}evi{\'c}~\cite{MR2355670}, the failure of $\square(\kappa)$ from the productivity of the $\kappa$-chain condition is due to Rinot~\cite{MR3271280}, the implication from a failure of $\square(\kappa)$ to the weak compactness of $\kappa$ in $\LL$ is a consequence of results of Jensen~\cite{MR0309729} and the equivalence of stationary reflection and weak compactness in canonical inner models $\LL[E]$ is due to Jensen~\cite{MR0309729} and Zeman~\cite{MR2563821}. All other implications and non-implications in the diagram are either trivial, or are proved in the current paper.  Note that the 3-way equivalence between ``$\kappa$ is weakly compact", ``$\exists \mathcal{F} ~ \left[\text{$\kappa$-c.c.} \subseteq \text{$\mathcal{F}$-layered}\right]$", and ``$\kappa$-c.c. $\subseteq \kappa$-stat.-layered" follows from Theorems \ref{theorem:WeaklyCompactEquivalentChainCondition} and \ref{thm_WeakCompactEquiv}.


\bibliographystyle{plain}
\bibliography{references}

\end{document}